\documentclass[]{article}
\usepackage[left=1in,right=1in, top=1in,bottom=1in]{geometry}     
\usepackage[dvipsnames]{xcolor}
\usepackage{xcolor}
\usepackage[english]{babel}
\usepackage{cite}
\usepackage[american,siunitx]{circuitikz}
\usepackage{color}
\usepackage[T1]{fontenc}
\usepackage{comment}
\usepackage[latin9]{inputenc}
\usepackage{amsmath}
\usepackage{amssymb}
 
\usepackage{amsthm} 
\usepackage{mathtools,cuted}
\usepackage{mathrsfs}
\usepackage{bbm}
\usepackage{cite}
\allowdisplaybreaks
\usepackage{lipsum}
\usepackage{multicol}
\usepackage{changepage}
\usepackage{pgfplots}
\usepackage{tikz}
\pgfplotsset{compat=1.7}
\usetikzlibrary{shapes.geometric,positioning}

\allowdisplaybreaks

\newtheorem{assumption}{\bf Assumption}

\newtheorem{thm}{\bf Theorem}
\newtheorem{lem}{\bf Lemma}
\newtheorem{proposition}{\bf Proposition}

\theoremstyle{definition}

\usepackage{dsfont}
\usepackage{relsize}
\usepackage{subdepth}
\usepackage{babel}
\usepackage{url}
\DeclareFontFamily{OT1}{pzc}{}
\DeclareFontShape{OT1}{pzc}{m}{it}{<-> s * [1.200] pzcmi7t}{}
\DeclareMathAlphabet{\mathpzc}{OT1}{pzc}{m}{it}

\newcommand{\LL}{\ensuremath{\mathcal{L}}}
\newcommand{\FF}{\ensuremath{\mathscr{F}}}
\usepackage{filecontents}

\usepackage{stfloats}

\renewcommand\footnotemark{}

\usepackage{bigints}

\usepackage{scalerel}

\usepackage[framemethod=tikz]{mdframed}

\newmdenv[
  hidealllines=true,
  backgroundcolor=blue!10,
  innerleftmargin=8pt,
  innerrightmargin=8pt,
  innertopmargin=0pt,
  innerbottommargin=6pt,
  leftmargin=-0pt,
  rightmargin=-0pt
]{shadedbox}

\definecolor{forest}{rgb}{0, 0.5, 0}

\makeatletter
\let\NAT@parse\undefined
\makeatother
\usepackage{hyperref}
\hypersetup{
    colorlinks=true,
    linkcolor=forest,
    filecolor=magenta,
    urlcolor=forest,
    citecolor=forest
}

\usepackage{anyfontsize}
\usepackage{cleveref}

\title{\LARGE
State-Output Risk-Constrained Quadratic Control \\ of
Partially Observed Linear Systems 
}

\author{Nikolas Koumpis$^{\star}$, Anastasios Tsiamis$^{\dagger}$ and Dionysios Kalogerias$^{\star}$
\thanks{$^{\star}$
Department of Electrical Engineering, Yale University, New Haven, CT 06511 (email: \{nikolaos.koumpis, dionysis.kalogerias\}@yale.edu).
}
\thanks{$^{\dagger}$ Automatic Control Laboratory, ETH Z\"urich, 8092, Switzerland (email: {atsiamis}@control.ee.ethz.ch).}
}

\makeatletter
\def\@cite#1#2{[\textbf{#1\if@tempswa , #2\fi}]}
\def\@biblabel#1{\textbf{[}\textbf{#1}\textbf{]}}
\makeatother
\begin{document}
\pgfmathdeclarefunction{gauss}{2}{\pgfmathparse{1/(#2*sqrt(2*pi))*exp(-((x-#1)^2)/(2*#2^2))}%
}

\date{}
\maketitle
\thispagestyle{empty}
\pagestyle{empty}
\begin{abstract}
We propose a methodology for performing risk-averse quadratic regulation of partially observed Linear Time-Invariant (LTI) systems, disturbed by process and output noise.
To compensate against the induced variability due to \textit{both} types of noises, state regulation is subject to two risk constraints. The latter render the resulting controller to be cautious of stochastic disturbances, by restricting the statistical variability, namely, a simplified version of the cumulative expected predictive variance, of both the \emph{state} and the \emph{output}.
It turns out that our proposed formulation results in an optimal risk-averse policy that preserves favorable characteristics of the classical Linear Quadratic (LQ) control.
In particular, 
the optimal policy has an affine structure with respect to the minimum mean square error (mmse) estimates. The linear component of the policy regulates the state more strictly in riskier directions, where the
process and output noise covariance, cross-covariance, and the corresponding penalties are simultaneously large. This is achieved by ``inflating" the state penalty in a systematic way. The additional affine terms force the state against pure and cross third-order statistics of the process and output disturbances. 
Another favorable characteristic of our optimal policy is that it
can be pre-computed off-line, thus, avoiding limitations of prior work. 
Stability analysis shows that the derived controller is always internally stable regardless of parameter tuning. 
The functionality of the proposed risk-averse policy is illustrated through a working example via extensive numerical simulations. 
\end{abstract}

\section{Introduction}
Decision policies designed to be optimal on average are often inadequate in many practical applications, especially where the system state transitions are subject to unexpected, less probable, though with possibly catastrophic consequences, random events. 
 Critical applications where such risky events should be accounted for are especially pronounced in many areas including robotics \cite{kim2019bi,pereira2013risk}, wireless communications \cite{ma2018risk}, networking \cite{bennis2018ultrareliable,iwaki2021multi}, control \cite{fan2020bayesian}, formal methods~\cite{lindemann2021stl}, health \cite{Cardoso2019}, and finance \cite{dentcheva2017statistical} to name a few.
{A major challenge that arises is that we not only have to deal with stochastic environments, but also with imperfect information, partial observability, and noisy measurements.
Although in some cases we can theoretically characterize optimal risk-aware policies for general partially-observed systems, their implementation can be really hard if not impossible~\cite{tsiamis2021linear}. In this paper, we are interested in \emph{implementable} risk-averse control policies that  counteract against those less probable though catastrophic events for both the states and the outputs, at the cost of slightly sacrificing performance on average.}
Recently, a risk-constrained formulation for Linear Quadratic (LQ) control of partially-observed linear systems was introduced in~\cite{tsiamis2021linear}. While the optimal risk-averse policy of~\cite{tsiamis2021linear} can be characterized, it is only implementable in the two simple settings of Gaussian noise and full state-observations. In the non-Gaussian case, implementing the optimal policy, requires online computation of a filtered version of the Riccati difference equation. Such a computation is hard if not impossible to perform. Here, inspired by~\cite{tsiamis2021linear}, we aim to find a control policy that can be implemented efficiently while it retains its risk-aversion properties. Our contributions are the following.

\textbf{State and output risk constraints.} We propose novel risk constraints, where we restrict the variability of both the state and the output. In particular, we restrict the cumulative expected predictive variance of the state and output, where the prediction is conditioned on the previous state. By adjusting the constraint specifications, we can control the tradeoff between average LQ cost and statistical variability due to \textit{both} process and output uncertainties. State predictive variance constraints were also studied in~\cite{tsiamis2021linear}, where the prediction was conditioned on the past output information. Here, the prediction is conditioned on the previous state, i.e. it discounts the uncertainty of the estimation error. However, this allows us to obtain a tractable and implementable risk-averse controller. In particular, the optimal policy here can be pre-computed \emph{off-line} based only on second and third order noise-statistics, overcoming the limitations of~\cite{tsiamis2021linear}.\\
\textbf{Optimal Risk-Aware policy \& Stability.} We show that the proposed risk-averse formulation results in a twice quadratically constrained LQ
problem, which admits a closed-form solution. The optimal risk-averse strategy is affine with respect to the state estimates, with the affine terms  directly repelling the pure and cross third order noise-statistics. Further, an inflated state penalty matrix amplifies regulation by over-measuring directions where process and output noise present (jointly and each separately) high risk.\\
\textbf{Arbitrary noise model.}  Our results are applicable for all skew-symmetric, heavy tailed, and  skewed (process and/or output) models provided that the corresponding fourth order moments remain finite. This is in contrast to the classical Linear Exponential Quadratic Gaussian control~\cite{whittle}, where the noise distributions have to be Gaussian.\\
\textbf{Separation.} We show that separation holds, in the sense that the optimal policy is affine with respect to the minimum mean square error (mmse) estimate. Hence, the filter and the controller can be designed separately. Computing efficiently the mmse estimate in the non-Gaussian is a very hard problem of independent interest and outside of the scope of this paper. In the simulations we use the Kalman Filter, which, despite being suboptimal, is the best linear filter~\cite{anderson2012optimal}.

\textsc{\textbf{Related work.}} Having its origins in mathmatical finance and operations research, risk-averse optimization has been lately naturally emerging in many applications, and has been considered in a variety of  contexts~\cite{A.2018,Cardoso2019,W.Huang2017,Jiang2017,Kalogerias2018b,Tamar2017,Vitt2018,Zhou2018,Ruszczynski2010,SOPASAKIS2019281,chapman2019cvar}. 
The main paradigm shift in risk-averse optimization is to replace expectations by more general \textit{risk measures} \cite{shapiro2021lectures}, as an attempt to more effectively capture the tail behavior of the involved random cost function, on top of or trading with mean performance. Typical examples are mean-variance functionals~\cite{Markowitz1952,shapiro2021lectures}, mean-semideviations~\cite{Kalogerias2018b}, and Conditional Value-at-Risk (CVaR)~\cite{Rockafellar1997}. In particular for the case of control systems, apart form the well-known and studied exponential quadratic approach, CVaR optimization techniques have also been considered for risk-aware constraint satisfaction~\cite{chapman2019cvar}. Although CVaR captures variability and tail events well, CVaR optimization problems rarely admit closed-form expressions \cite{chapman2019cvar}, \cite{chapman2021classical}. 
\par
The predictive variance of the stage cost was introduced in \cite{tsiamis2020risk}, \cite{abeille2016lqg} for risk-aware quadratic state regulation in the fully observable case and subsequently it was applied in \cite{tsiamis2021linear} for partially observed systems.  
Although both formulations characterize the resulting policy in closed form, the latter ends up being dependent on filtered information and as such it can be pre-computed off-line only when both the process and the measurement noise are Gaussian.
Finally, a closely related problem is the celebrated Linear Exponential Quadratic Gaussian (LEQG) control problem~\cite{whittle,Speyer2008STOCHASTIC}. Although it admits a closed-from solution, it also requires Gaussian noise, which does not capture distributions with asymmetric (skewed) structure. Moreover, tuning the exponential parameter can be challenging, since certain values above a threshold lead to unstable controllers (neurotic breakdown).
\par
\textsc{\textbf{Notation.}} With  $\mu_{s,o}$ we refer to the pair $(\mu_{s},\mu_{o})$, and by  $\mu_{s,o}>0$ we declare that both $\mu_{s}$ and $\mu_{o}$ are positive. 
We denote a sequence of arrays $x_{0},\dots,x_{k}$, as $x_{0:k}$, and lastly the $\sigma-$algebra generated by a random array $x$ as $\sigma(x)$. We use $||x||^{2}_{Q}$ to denote the $Q-$weighted norm of $x\in\mathbb{R}^{n}$, with the matrix $Q\in\mathbb{R}^{n\times{n}}$ being positive (semi) definite.
\section{Risk-constrained Linear Quadratic formulation with partial observations}
Consider a dynamical system described by the linear difference equation\vspace{-1pt}
\begin{equation}
    x_{t+1}=Ax_{t}+Bu_{t}+w_{t+1},\label{dynamics}
\end{equation}
with its output given by\vspace{-2pt}
\begin{equation}
    y_{t}=Cx_{t}+\epsilon_{t}.\label{out}
\end{equation}
Here $x_{t}\in\mathbb{R}^{n}$ is the hidden state of the system, $u_{t}\in\mathbb{R}^{m}$ is the control input, and $y_{t}\in \mathbb{R}^{q}$ are the outputs/observations of the system. In our study, the model uncertainty $w_{t}$ and the output noise $\epsilon_{t}$ are assumed \textit{i.i.d} stochastic sequences, and not necessarily Gaussian, with $\mathbb{E}\{w_{t}\}=\bar{w}$, $\mathbb{E}\{\epsilon_{t}\}=\bar{\epsilon}$, $\mathbb{E}\{(w_{t}-\bar{w})(w_{t}-\bar{w})^{\top}\}=W$, and $\mathbb{E}\{(\epsilon_{t}-\bar{\epsilon})(\epsilon_{t}-\bar{\epsilon})^{\top}\}=E$, $\mathbb{E}\big\{(w_{t}-\bar{w})(\epsilon_{t}-\bar{\epsilon})^{\top}
\big\}=H$. 
Let $\mathscr{F}_{t}\triangleq \sigma(y_{0:t},u_{0:t-1})$ be the sigma algebra generated by all the observable  quantities up to time $t$.
Based on the above notation, the state estimate and state prediction at time $t$ respectively read 
\begin{align}
\widehat{x}_{t|t-1}&=\mathbb{E}\{x_{t}|\mathscr{F}_{t-1}\}, ~~~~~
\widehat{x}_{t|t}=\mathbb{E}\{x_{t}|\mathscr{F}_{t}\}\label{estimates},
\end{align}
with corresponding errors
\begin{align}
    \rho_{t}&={x}_{t}-\widehat{x}_{t|t-1},~~~~~~
    e_{t}={x}_{t}-\widehat{x}_{t|t},\label{est}
\end{align}
where at each time we expect zero for both, i.e.,  $\mathbb{E}\{e_{t}|\mathscr{F}_{t-1}\}=\mathbb{E}\{\rho_{t}|\mathscr{F}_{t-1}\}=0$. Lastly, we also define 
\begin{align}
    \Sigma_{t|t}&=\mathbb{E}\{e_{t}e^{\top}_{t}|\mathscr{F}_{t}\},\label{ercov}\\[5pt]
    \Sigma_{t|t-1}&=\mathbb{E}\{e_{t}e^{\top}_{t}|\mathscr{F}_{t-1}\},\label{predcov}\\[5pt]
    H_{t|t-1}&=\mathbb{E}\{e_{t}\rho^{\top}_{t}|\mathscr{F}_{t-1}\}\label{croscov}.
\end{align}
Throughout this paper, we assume that both the process and output noises have finite fourth moments. 
\begin{assumption}[\textbf{Noise Regularity}]\label{as1} The processes $w_{t}$, and $\epsilon_{t}$ have finite fourth-order moments, i.e., for every $t\in\mathbb{N}$, $\mathbb{E}\{||w_{t}||^{4}_{2}\}<\infty$, and $\mathbb{E}\{||\epsilon_{t}||^{4}_{2}\}<\infty$. 
\end{assumption}
The above assumption guarantees that our risk-averse LQ control problem, defined below, is well-posed. It covers general, potentially non-Gaussian, distributions.
Before we state the main problem, let us define the ``extended" sigma algebra $\overline{\mathscr{F}}_{t}\triangleq \sigma(y_{0:t},u_{0,t-1},x_{0:t})$, which is the sigma algebra of the observable quantities plus the hidden states. Let also $\LL_2(\mathscr{F}_{t})$ be the space of square integrable vectors, which are $\mathscr{F}_{t}-$measurable. We pose the following problem:
\begin{align} 
J(u)=&\min _{u} \mathbb{E}\left\{ ||x_{N}||^{2}_{Q}+\sum_{t=0}^{N-1} ||x_{t}||^{2}_{Q}+||u_{t}||^{2}_{R}\right\}\nonumber \\\text { s.t. }
J_{s}(u)=&\mathbb{E}\left\{\sum_{t=1}^{N}\left[||x_{t}||^{2}_{Q_{s}}{-}\mathbb{E}\left(||x_{t}||^{2}_{Q_{s}} \mid \overline{\mathscr{F}}_{t-1}\right)\right]^{2}\right\} \leq \epsilon_{s},
\nonumber\\
J_{o}(u)=&\mathbb{E}\left\{\sum_{t=1}^{N}\left[||y_{t}||^{2}_{Q_{o}}{-}\mathbb{E}\left(||y_{t}||^{2}_{Q_{o}} \mid \overline{\mathscr{F}}_{t-1}\right)\right]^{2}\right\} \leq \epsilon_{o},
\nonumber\\
&x_{t+1}=Ax_{t}+Bu_{t}+w_{t+1},
\nonumber\\
&y_{t}=Cx_{t}+\epsilon_{t},
\nonumber\\
& u_{t} \in \mathcal{L}_{2}\left(\mathscr{F}_{t}\right), t=0, \ldots, N-1.\label{problem}
\end{align}
We impose constraints on the cumulative expected predictive variance of the state and the output. This essentially forces the optimal controller to not only optimize average performance but also reduce the \emph{variability} of the state and the output. Note that we \emph{discount} the risk of the estimation error $e_t$ in the predictive variance. This is why the prediction in the constraints is conditioned on the \emph{extended} sigma algebra $\overline{\mathscr{F}}_{t-1}$ and not $\FF_{t-1}$. Taking the estimation error into account, would result in an intractable optimal control problem~\cite{tsiamis2021linear}. By discounting $e_t$, we sacrifice some risk-aversion to obtain an implementable controller. However, we still account for the immediate process and output noises.  In any case, the control input only has access to the partial observations $\mathscr{F}_{t-1}$ as captured by the requirement $u_t\in\LL_2(\FF_{t})$.  Finally, the positive semi-definite weighting matrices $Q_s,Q_o$ are some additional tuning parameters that we can tweak, offering more flexibility to our LQ formulation. 

\section{Optimal Risk-Averse LQR Controllers}
The derivation of the optimal risk-averse policy tracks \cite{tsiamis2020risk} and it is summarized in the following steps: We first show that \eqref{problem} is well defined and that it can be recast to a sequential \textit{twice quadratically constrained quadratic program (TQCQP)}. Then, we utilize Lagrangian duality to solve \eqref{problem} in closed form.
\begin{proposition}[\textbf{TQCQP Reformulation}]
\label{prop1}
Let Assumption \ref{as1} be in effect, and define 
\begin{align}
    p_{t}\triangleq{C}(w_{t}-\bar{w})+\epsilon_{t}-\bar{\epsilon},\nonumber
\end{align}
along with the higher order weighted statistics
\begin{align}
    \mathrm{M}_{w}&\triangleq\mathbb{E}\Big\{(w_{t}-\bar{w})(w_{t}-\bar{w})^{\top}Q_{s}(w_{t}-\bar{w})\Big\}\nonumber,\\
     \mathrm{M}_{\epsilon}&\triangleq\mathbb{E}\Big\{(\epsilon_{t}-\bar{\epsilon})(\epsilon_{t}-\bar{\epsilon})^{\top}Q_{o}(\epsilon_{t}-\bar{\epsilon})\Big\}\nonumber,\\[5pt]
    m_{w}&\triangleq \mathbb{E}\Big\{\Big((w_{t}-\bar{w})^{\top}Q_{s}(w_{t}-\bar{w})-\mathrm{tr}(Q_{s}W)\Big)^{2}\Big\},\nonumber\\
    \mathrm{M}&\triangleq\mathbb{E}\Big\{p_{t}p_{t}^{\top}Q_{o}p_{t}\Big\}\label{enea},\\
    m_{w\epsilon}&\triangleq \mathbb{E}\Big\{\Big(p_{t}^{\top}Q_{o}p_{t}-\mathrm{tr}(Q_{o}P)\Big)^{2}\Big\}.\nonumber
\end{align}
Then, the risk-constrained LQ problem \eqref{problem} is well-defined and equivalent to the sequential variational TQCQP
\begin{align} 
\min_{u}&~ \mathbb{E}\left\{x_{N}^{\top} Q x_{N}+\sum_{t=0}^{N-1} x_{t}^{\top} Q x_{t}+u_{t}^{\top} R u_{t}\right\}\nonumber\\
\text{\normalfont s.t.}&~\mathbb{E}\left\{\sum^{N}_{t=1}4x_{t}^{\top}Q_{s}WQ_{s}x_{t}+4x^{\top}_{t}Q_{s}\mathrm{M}_{w}\right\}<\bar{\epsilon}_{s},\nonumber
\\
&~\mathbb{E}\left\{\sum^{N}_{t=1}4x^{\top}_{t}C^{\top}Q_{o}PQ_{o}Cx_{t}+4x^{\top}_{t}C^{\top}Q_{o}\mathrm{M}\right\}<\bar{\epsilon}_{o},\nonumber
\\
&x_{t+1}=Ax_{t}+Bu_{t}+w_{t+1},\nonumber
\\
& u_{t} \in \mathcal{L}_{2}\left(\mathscr{F}_{t}\right), t=0, \ldots, N-1 ,\label{qcqp}
\end{align}
\textit{where}
\begin{align}
P&{=}CWC^{\top}+CH+(HC)^{\top}+E, \nonumber\\[5pt]
\bar{\epsilon}_{s}&{=}\epsilon_{s}-Nm_{w}+4N(\mathrm{tr}(Q_{s}W))^{2},\nonumber\\[5pt]
\bar{\epsilon}_{o}&{=}\epsilon_{o}{-}Nm_{w\epsilon}{-}4N\bar{\epsilon}^{\top}Q_{o}\mathrm{M}{-}8N\mathrm{tr}(C^{\top}Q_{o}PQ_{o}CW)-4N\mathrm{tr}(Q_{o}PQ_{o}Z)\nonumber,
\end{align}
and $Z=\mathbb{E}\big\{(\bar{\epsilon}-C\delta_{t})(\bar{\epsilon}-C\delta_{t})^{\top}
\big\}$.
\end{proposition}
The proof is given in the Appendix \ref{proofprop1}. \textit{Proposition \textbf{\ref{prop1}}} allows us to utilize duality theory and thus solve \eqref{problem} in closed form via dynamic programming. 
\subsection{Lagrangian Duality}
To tackle problem \eqref{problem}, we consider the \textit{variational Lagrangian} $\mathcal{L}:\mathcal{L}_{2}(\mathscr{F}_{0})\times\dots\times\mathcal{L}_{2}(\mathscr{F}_{N-1})\times\mathbb{R}_{+}\times\mathbb{R}_{+}\rightarrow\mathbb{R}$ of the sequential TQCQP \eqref{qcqp}, defined as
\begin{align}
    \mathcal{L}(u;\mu_{s,o})\triangleq J(u)+\mu_{s}J_{s}(u)+\mu_{o}J_{o}(u)-\mu_{s}\bar{\epsilon}_{s}-\mu_{o}\bar{\epsilon}_{o},\label{lag}
\end{align}
where $\mu_{s,o}\geq0$ are multipliers associated with the variational risk constraints of \eqref{qcqp}. The dual function $\mathcal{D}:\mathbb{R}^{+}\times\mathbb{R}^{+}\rightarrow[-\infty,+\infty)$ of the primal problem \eqref{qcqp} is defined as 
\begin{align}
\mathcal{D}(\mu_{s,o})=\inf_{u\in\mathcal{U}}  \mathcal{L}(u,\mu_{s,o}),\nonumber
\end{align}
where
\begin{align}
\mathcal{U} \triangleq\left\{u_{0: N-1} \in \prod_{t=k}^{N-1} \mathcal{L}_{2}\left(\mathscr{F}_{t}\right) \Bigg| x_{t+1}=A x_{t}+B u_{t}+w_{t+1}\right\}\nonumber
\end{align}
refers to the constraints that remain not dualized. The optimal value of the always concave dual problem 
\begin{align}
\sup_{\mu_{s,o}\geq0}\mathcal{D}(\mu_{s,o})\equiv\sup_{\mu_{s,o}\geq0}\inf_{u\in\mathcal{U}}\mathcal{L}(u,\mu_{s,o}),\label{dual}
\end{align}
with $\mathcal{D}^{*}=\sup_{\mu_{s,o}\geq0}\mathcal{D}(\mu_{s,o})\in[-\infty,\infty]$ is the tightest under-estimate of the optimal value of the primal problem $J^{*}$, when knowing only $\mathcal{D}$. Leveraging Lagrangian duality, we may now state the following
result, which provides sufficient optimality conditions for the TQCQP \eqref{qcqp}.
\begin{thm}[\textbf{Optimality Conditions}]\label{thm1} With \textit{Assumption} \ref{as1} being in effect, suppose there exist a feasible policy-multiplier pair $(u^{*},\mu^{*}_{s,o})\in\mathcal{U}\times\mathbb{R}_{+}\times\mathbb{R}_{+}$ such that 
\begin{enumerate}
    \item $\mathcal{L}(u^{*}(\mu^{*}_{s,o}),\mu^{*}_{s,o})=\min_{u\in\mathcal{U}}\mathcal{L}(u,\mu^{*}_{s,o})=\mathcal{D}(\mu^{*}_{s,o})$
    \item $J_{s}(u^{*})\leq\bar{\epsilon}_{s}$, and $J_{o}(u^{*})\leq\bar{\epsilon}_{o}$.
    \item $\mu_{s}(J_{s}(u^{*})-\bar{\epsilon}_{s})=0$, and $\mu_{o}(J_{o}(u^{*})-\bar{\epsilon}_{o})=0$.
\end{enumerate}
Then, $u^{*}$ is optimal for the primal \eqref{qcqp} and the initial problem \eqref{problem}, $\mu_{s,o}$ are optimal for the dual problem \eqref{dual}, and \eqref{qcqp} exhibits zero duality gap.
\end{thm}
The proof of \textit{Theorem \ref{thm1}} is omitted as it follows as direct application of Theorem 4.10~in~\cite{ruszczynski2011nonlinear}. Within our context,  \textit{Theorem \ref{thm1}} provides justification of the Lagrangian relaxation approach we take hereafter in regard to solving Problem \eqref{problem}. In particular, by choosing certain values for the multipliers $\mu_{s,o}$, evaluation of the constraints of \eqref{problem} determine respective values for $\bar{\epsilon}_{s,o}$, which in most cases satisfy the conditions of \textit{Theorem \ref{thm1}}. In other words, \textit{Theorem \ref{thm1}} may be used `in reverse', as a verification device, or even to tabulate certain tolerance/multiplier pairs which lead to good performance. Also note that, for every optimal solution of the corresponding Lagrangian \eqref{lag}, the resulting constraint values can be efficiently evaluated recursively using standard LQ theory, along the lines of \cite{tsiamis2021linear}.
\subsection{Optimal risk-averse control policies}
Let $\mu_{s,o}>0$ be arbitrary but fixed. First, we may simplify the Lagrangian $\mathcal{L}$ and express it within a canonical
dynamic programming framework. In this respect, we have
the following result.
\begin{lem}[\textbf{Lagrangian Reformulation}]
For every $u_{t}\in\mathcal{L}_{2}(\mathscr{F}_{t})$, $t\leq{N-1}$, the Lagrangian function $\mathcal{L}$ can be expressed as
\begin{equation}
\begin{aligned}
\mathcal{L}(u;\mu_{s,o})&{=} \mathbb{E}\left\{{g_{N}(x_{N};\mu_{s,o})}{+}\sum_{t=0}^{N-1} g_{t}(x_{t},u_{t};\mu_{s,o})\right\}{+}g_{\mu_{s,o}}\\ \text {\normalfont s.t. }&x_{t+1}=Ax_{t}+Bu_{t}+w_{t+1},
\\
& u_{t} \in \mathcal{L}_{2}\left(\mathscr{F}_{t}\right), t=0, \ldots, N-1 ,
\end{aligned}
\label{problem3}
\end{equation}
where
\begin{align}
g_{N}(x_{N};\mu_{s,o})&\triangleq x_{N}^{\top}Q_{\mu_{s,o}} x_{N}+x^{\top}_{N}\mathrm{M}_{\mu_{s,o}},\nonumber\\[5pt]
g(x_{t},u_{t};\mu_{s,o})&\triangleq x_{t}^{\top}Q_{\mu_{s,o}}x_{t}+x^{\top}_{t}\mathrm{M}_{\mu_{s,o}}+u^{\top}_{t}Ru_{t},\nonumber\\[5pt]
g_{0}(x_{0};\mu_{s,o})&\triangleq-x^{\top}_{0}(Q_{\mu_{s,o}}-Q)x_{0}-x^{\top}_{0}\mathrm{M}_{\mu_{s,o}},\nonumber
\end{align}
and 
\begin{align}
\hspace{-15pt}Q_{\mu_{s,o}}&\triangleq Q+4\mu_{s}Q_{s}WQ_{s}+4\mu_{o}C^{\top}Q_{o}PQ_{o}C\nonumber\\[5pt]
\hspace{-10pt}\mathrm{M}_{\mu_{s,o}}&\triangleq4\mu_{s}Q_{s}\mathrm{M}_{w}+4\mu_{o}(C^{\top}Q_{o}\mathrm{M}+2C^{\top}Q_{o}PQ_{o}\bar{\epsilon}),\label{mstat}
\end{align}
where 
\begin{align}
\mathrm{M}=\mathrm{M}_{w\epsilon}+\mathrm{M}_{\epsilon}\label{me}
\end{align}
\end{lem}
\begin{proof}
It follows from \textit{Proposition \ref{prop1}} and the form of $\mathcal{L}$.
\end{proof}

\begin{thm}[\textbf{Optimal Risk-averse Controls}] \label{th2}
Let Assumption \ref{as1} be in effect, and opt $\mu_{s,o}>0$. Then for all $t\in \mathscr{I}$, $\mathscr{I}=\{1,\dots,N\}$, $N\in\mathbb{N}$,
the optimal cost-to-go $\mathcal{L}^{*}(\mathscr{F}_{t};\mu_{s,o})$ can be expressed as 
\begin{align}
\mathcal{L}^{*}(\mathscr{F}_{t};\mu_{s,o})&=\widehat{x}^{\top}_{t|t}V_{t}\widehat{x}_{t|t}+2\widehat{x}^{\top}_{t|t}T_{t}\bar{w}+\widehat{x}^{\top}_{t|t}S_{t}\mathrm{M}_{\mu_{s,o}}+p_{t}.\nonumber
\end{align}
The risk-averse optimal controller reads
\begin{align}
\hspace{-7pt}u^{*}_{t-1}(\mu_{s,o})&=K_{t-1}\widehat{x}_{t-1|t-1}+h_{t-1}+l_{t-1},\label{policy}
\end{align}
where 
\begin{align}
K_{t-1}&=-\left(B^{\top} V_{t} B+R\right)^{-1}B^{\top}V_{t}A,\nonumber\\
h_{t-1} &=-\left(B^{\top} V_{t} B+R\right)^{-1} B^{\top}\left(V_{t}+T_{t}\right) \bar{w}\nonumber \\
l_{t-1} &=-\frac{1}{2}\left(B^{\top} V_{t} B+R\right)^{-1} B^{\top} S_{t} \mathrm{M}_{\mu_{s,o}}.\nonumber
\end{align}
The following backward recursions
\begin{align}
 V_{t-1}&=A^{\top}V_{t}A+Q_{\mu_{s,o}}-A^{\top}V_{t}B(B^{\top}V_{t}B+R)^{-1}B^{\top}V_{t}A\nonumber\\
 K_{t-1}&=-(B^{\top}V_{t}B{+}R)^{-1}B^{\top}V_{t}A\nonumber\\
 T_{t-1}&=(V_{t}+T_{t})(A+BK_{t-1})\nonumber\\
 S_{t-1}&=S_{t}(A+BK_{t-1})+I\nonumber\\
 p_{t-1}&=p_{t}-
 (h_{t-1}+l_{t-1})^{\top}(B^{\top}V_{t}B+R)(h_{t-1}+l_{t-1})\nonumber\\[5pt]
 &\hspace{10pt}+\bar{w}^{\top}(V_{t}+2T_{t})\bar{w}+\mathrm{M}^{\top}S_{t}\bar{w}+\mathrm{tr}(Q_{\mu_{s,o}}{\Sigma}_{t|t})\nonumber\\[5pt]
 &\hspace{10pt}-2\mathrm{tr}(V_{t}H_{t|t-1})+\mathrm{tr}(A^{\top}V_{t}A\Sigma_{t-1|t-1})+\mathrm{tr}(V_{t}W)\nonumber,
\end{align}
hold, starting from $V_{N}{=}Q_{\mu_{s,o}}$, $S_{N}{=}I$, $T_{N}{=}0$, and $p_{N}{=}\mathrm{tr}(Q_{\mu_{s,o}}\Sigma_{N})$. 
\end{thm}
Before sketching the proof of \textit{Theorem \textbf{\ref{th2}}} we provide the following important lemma
\begin{lem}\label{lem2}
The estimation error $e_{t}$ and the prediction error $\rho_{t}$ are statistically independent of the controls.
\end{lem}

\begin{proof}(\textit{Lemma \textbf{\ref{lem2}}})
See Appendix \ref{lemma2appendix}.
\end{proof}
\begin{proof}(\textit{Theorem \textbf{\ref{th2}}}) See \textit{Appendix} \ref{th2appendix}.
\end{proof}
We note that the separation principle holds here, in the sense that
the optimal control policy is affine w.r.t. the mmse state estimates. The policy parameters can thus be computed independently of the state estimate. In fact, the term $l_{t}$ forces the state to follow  $\mathrm{M}_{\mu_{s,o}}$ which, according to \eqref{mstat},  \eqref{me}, combines $\mathrm{M}_{w}$, $\mathrm{M}_{w\epsilon}$, and $\mathrm{M}_{\epsilon}$. While $\mathrm{M}_{w}$ accounts only for the process noise risky events, $\mathrm{M}_{w\epsilon}$ pre-compensates for the cross third order statistics between the output and process noise, the latter as seen from the output. Further, $\mathrm{M}_{\epsilon}$ considers the pure third order statistics of the output noise.
At this point, we note that the first constraint itself might not be sufficient to account for the variability of the output. In particular, as seen by the form of the optimal controller, the second constraint is required in order to counteract the output noise (term $M_{\epsilon}$ in $M$).
Both constraints induce an extra inflation to the state penalty modifying the Riccati equation and subsequently the policy gain accordingly. As a result, the state is regulated more strictly in riskier directions, where the
process and output noise covariance, cross-covariance, and the corresponding penalties are simultaneously large.
In the following, we show that the optimal controller~\eqref{policy} is internally stable under the following standard assumption.
\begin{assumption}\label{as2}
The pair $(A,B)$ is stabilizable, the matrix $Q$ is positive semi-definite and the matrix $R$ is positive definite.
\end{assumption}

\begin{proposition}
[\textbf{Stability}]
Let Assumptions \ref{as1} and \ref{as2} be in effect. For a fixed pair $\mu_{s,o}\geq0$ consider the control policy $u^{*}(\mu_{s,o})$, as defined in \eqref{policy}. As $N\rightarrow\infty$, $V_{t}$ converges exponentially fast to the unique stabilizing solution of the algebraic Riccati equation.
\begin{align}
&V=A^{\top}VA+Q_{\mu_{s,o}}-A^{\top}VB(B^{\top}VB+R)^{-1}B^{\top}VA.\nonumber
\end{align}
Thus, for every $t\geq0$ it is true that, as $N\rightarrow\infty$,
\begin{align}
K_{t}&\rightarrow{K}\triangleq-(B^{\top}VB+R)^{-1}B^{\top}VA,\nonumber\\[5pt]
S_{t}&\rightarrow{S}\triangleq(I-(A+B K))^{-1},\nonumber\\[5pt]
T_{t}&\rightarrow{T}\triangleq{V}(A+BK)(I-(A+BK))^{-1},\nonumber\\[5pt]
h_{t-1}&\rightarrow{h}\triangleq -\left(B^{\top} V B+R\right)^{-1} B^{\top}\left(V+T\right) \bar{w},\nonumber \\[5pt]
l_{t-1}&\rightarrow{l}\triangleq -\frac{1}{2}\left(B^{\top} V B+R\right)^{-1} B^{\top} S \mathrm{M}_{\mu_{s,o}}.\nonumber
\end{align}
\end{proposition}
\begin{proof}
Since $Q>Q_{\mu_{s,o}}$, detectability of $(A,Q^{1/2})$ implies detectability of $(A,Q_{\mu_{s,o}}^{1/2})$. That $A+BK$ is Hurwitz follows from standard LQR theory \cite{anderson2012optimal}.
\end{proof}
{Note that the above result only concerns the internal stability of the policy. For the closed-loop system to be well-behaved, we also need the mmse estimator, or any other state estimator, to exhibit bounded mean square estimation error. This can be achieved for example if $(A,Q^{1/2})$ is detectable.}

\section{Simulations and results}
Consider the inverting operational amplifier $S_{c}$ shown in Fig \ref{opamp}
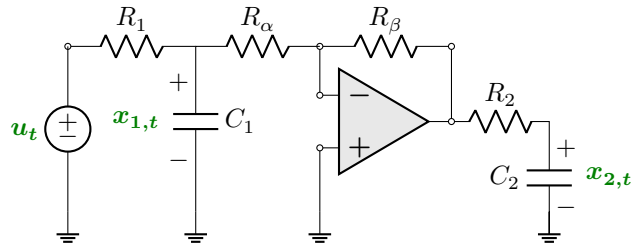
\begin{figure}[h!]
\centering
\begin{circuitikz}[american]
\ctikzset{bipoles/length=1cm}
\draw (0, 0) node[op amp,fill=gray!18] (opamp) {}
(opamp.-) to[short,o-o] (-0.85, 1.0)to[R,l^=$R_{\beta}$](0.9, 1.0)to[short,o-o] (0.9,0)to[R,l^=$R_{2}$](2.2,0)to[C,l_=$C_{2}$,v^=\textcolor{forest}{$\boldsymbol{x_{2,t}}$}](2.2,-1.5)to (2.2,-1.2)node[ground]{};
\draw (-0.85, 1.0)to[R,l_=$R_{\alpha}$](-2.5,1)to
[C,l^=$C_{1}$,v=\textcolor{forest}{$\boldsymbol{x_{1,t}}$}](-2.5,-1)to (-2.5,-1.2)node[ground]{};
\draw (-2.5,1)to[R,l_=$R_{1}$,-o](-4.2,1)
(opamp.+) to[short,o-] (-0.85,-1.2) node[ground]{};
\draw(-4.2,-1.2)node[ground]{}to[V,l^=\textcolor{forest}{$\boldsymbol{u_{t}}$},invert](-4.2,1){};
\end{circuitikz}
\caption{Inverting operational amplifier}
\label{opamp}
\end{figure}
which may be described by the second order state space model 
\begin{align}
\hspace{-10pt}
x_{t+1}&=\left[\begin{array}{cc} 0.172 & 0 \\ 1.046 & 0.8869 \end{array}\right]x_{t}+\left[\begin{array}{c} 0.1882\\0.2762 \end{array}\right]u_{t}+w_{t},\label{discrete}
\end{align}
where we followed \cite[p.~92]{murray} with $R_{1}=0.5$, $R_{2}=1$, $R_{\alpha}=5$, $R_{\beta}=0.25$, $C_{1}=0.5$, $C_{2}=0.3$, and subsequently discretized with $T_{s}=0.4~\mathrm{sec}$. The system state $x=(x_{1},x_{2})$ comprises the two capacitor voltages denoted in Fig. \ref{opamp} by forest green.
 

  \begin{figure}[h!]
      \centering
  \begin{tikzpicture}[
    TFblock/.style= {
        draw,  
        minimum size=1.5cm}]
\node [TFblock,fill=gray!20,] (a) at (0,0) {$S_{i}$};

\node (rect) at (2,1) [draw=black,thick,minimum width=6cm,minimum height=4.5cm,label={[black]30: $S$}] {};

\node[circle,
   draw=black,
   text=black,
   fill=gray!20,
   minimum size=0.01cm
   ] (b) at (2,0) {$+$};
   
\node [TFblock,fill=gray!20,] (c) at (4,0) {$S_{c}$};
   
 \node[circle,
   draw=black,
   text=black,
   fill=gray!20,
   minimum size=0.01cm
   ] (d) at (6,0) {$+$};

\node [TFblock,fill=gray!20,] (f) at (8,0) {$S_{o}$};

\node (g) at (2,2.5) {
\begin{tikzpicture}[scale=0.2]
\begin{axis}[samples=50, xticklabels={}, yticklabels={},
height=6cm, width=15cm,line width=1.5, axis on top, grid = major]
\addplot [fill=white!20, draw=forest,line width=3]
{0.8*gauss(-2,0.5)+0.2*gauss(3,0.8)}
\closedcycle;
\end{axis}
\end{tikzpicture}};

\node (s) at (6,-2.5) {
\begin{tikzpicture}[scale=0.2]
\begin{axis}[samples=50, xticklabels={}, yticklabels={},
height=6cm, width=15cm,line width=1.5, axis on top, grid = major]
\addplot [fill=white!20, draw=forest,line width=3] 
{0.8*gauss(-2,0.5)+0.2*gauss(3,0.8)}\closedcycle;
\end{axis}
\end{tikzpicture}};

\draw [-latex](a.east)edge node[above] {$u_{}$} (b.west)
    (b.north) ++(0,1.5)edge 
    node[right] {$w_{}$} (b.north)
    (b.east) edge 
    node[above] {} (c.west)
    (c.east) edge 
    node[above] {$v_{o}$} (d.west)
    (d.east) edge 
    node[above]{$y$} (f.west)
    (f.west) edge 
    node[above]{} (d.east)
    (s.north) edge 
    node[right]{$\epsilon$} (d.south)++(0,1.53);
\end{tikzpicture}
\caption{Interconnection of inverting op-amp with two terminal stages. The input stage $S_{i}$ refers to $S_{c}$ which has to follow $u$ and transmit $v_{o}$ despite the voltage socks $w_{t}$ and the output disturbances $y_{t}$.}
\label{d1}
  \end{figure}
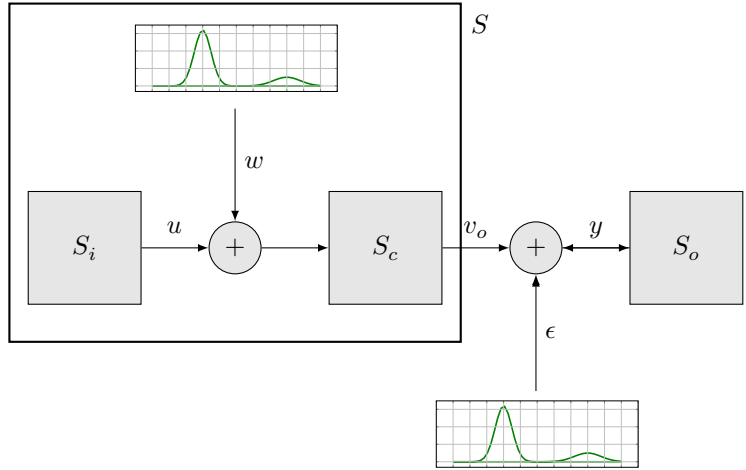
  
Based on the diagram in Fig. \ref{d1}, $S_{c}$ is internally affected from $S_{i}$ through $u$ and it is the part of $S$ that interfaces $S_{o}$ through the output $y$. The voltage of the resistor $R_{2}$ is set as an output and it is also observed for state estimation according to 
\begin{align}
y_{t}&=\left[\begin{array}{cc} 0.05 & -1\end{array}\right]x_{t}+\epsilon_{t},\label{sout}
\end{align}
where $\epsilon_{t}$ expresses random events that occur peripherally around the circuit, in this case in $R_{2}$.
\par 
The objective is to control the average energy storage in $S$ by regulating the internal state to certain voltage levels $x^{*}{=}(0.2117, 0.43995)$ while accounting for unexpected (internal) voltage socks. On top of that, we are interested in maintaining a low output variability subject to peripheral, rare but catastrophic events. Roughly speaking, we are interested in controlling the energy flow in $S$ in a bidirectional manner. On the one hand we  safeguard against risky events that act internally as disturbances in $u_{t}$, but also keep $S$ externally shielded from the unexpected events that occur it the output $y_{t}$, peripherally around $S$. 
\par
The stochastic disturbances $w_{t}$ model, random events within $S$. Assuming nominal operating conditions for $S_{i}$, $S_{c}$ receives symmetric random deviations around $u$, expressed by a Gaussian distribution $w_{t}\sim\mathcal{N}(0,0.1)$. On the contrary, unexpected-risky events that take place in 
$S$, are reflected on $w_{t}$ as `voltage socks' thus modelled as a mixture of two Gaussians $\mathcal{N}(0,0.01)$, and $\mathcal{N}(10,0.001)$ with weights $p_{1}=0.8$ and $p_{2}=0.2$, respectively. Further, we consider two different cases regarding the output disturbance $\epsilon_{t}$. Nominal operating conditions are modeled by $\epsilon_{t}\sim\mathcal{N}(0,0.01)$, while the Gaussian mixture 
$0.7\mathcal{N}(0,0.01)+0.3\mathcal{N}(20,0.005)$ indicates rare but highly undesirable events that reflect on $\epsilon_{t}$ as voltage shocks.
The reader may verify that the system \eqref{discrete}-\eqref{sout} is controllable and observable.
\par
Throughout the simulations we maintain $Q=\mathrm{diag}(1,1)$ and $R{=}1$  for the LQ stage-cost, while we set $Q_{s}{=}Q_{o}{=}\mathrm{diag}(1, 0.1)$ for the state and output risk-constraint. In all cases state estimates are obtained sequentially via a Kalman filter. Lastly, we compare the performance of our controller with the LEQG controller-filter system \cite{whittle}. Lastly, the tracking objective may be recast as regulation objective w.r.t. 
\begin{align}
    \Delta{x}_{t+1}=A\Delta{x}_{t}+B\Delta{u},
\end{align}
where $\Delta{x}_{t}=x_{t}-x^{*}$, $\Delta{u}_{t}=u_{t}-u^{*}$, and $u^{*}$ is such that
$x^{*}=Ax^{*}+Bu^{*}$. For the given $x^{*}$ we choose $Bu^{*}=B$, while the output targets to $y^{*}=-0.4294$. In all cases, a Kalman filter takes care of the mmse state estimation. The functionality of our controller is demonstrated by the following three operating scenarios:
\subsubsection{Case 1. (Skewed process noise-Gaussian output noise)}
In this scenario our risk-averse controller safeguards against internal risky events, while the surrounding conditions, are assumed to be nominal. 
Fig. \ref{case1} (top) depicts  $||x_{t}-x^{*}||^{2}_{Q}$ over the first $100$ time-steps, while at the bottom we plot the energy of the internal voltage socks. 
\begin{figure}[h!]
\centering
\includegraphics[scale=0.17]{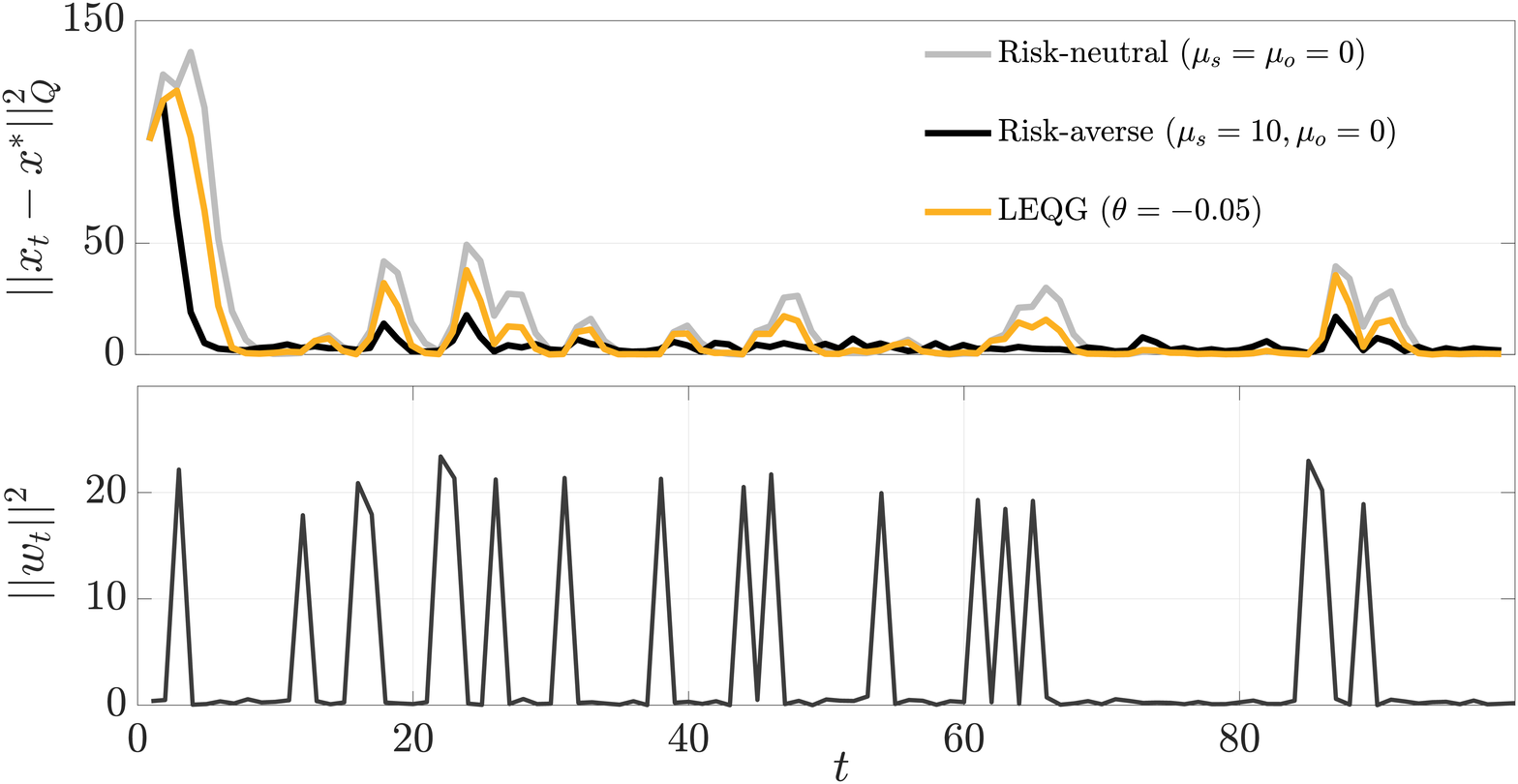}
\caption{State penalty variation $||x_{t}-x^{*}||^{2}_{Q}$ over time considering skewed process noise prior and Gaussian output noise for $S$ (top). Energy of internal voltage shocks over time (bottom).}
\label{case1}
\end{figure}
Note how the risk averse controller with $\mu_{s,o}=(10,0)$ drastically limits the variability of the stored energy, thus protecting $S$ from the internal voltage socks especially from those over $t\in (15,25)\cup(45,50)\cup(60,70)\cup(85,95)$.
\par
Although the LEQG controller performs better compared to the risk-neutral control, it seems to be more susceptible to the incoming voltage socks but with a slightly better average performance compared to the risk-averse controller (black). 
Roughly speaking, the LEQG policy is linear; it does not include affine terms related to third-order statistics and it renders the controller cautious to risky events by only magnifying the Gaussian process noise covariance. Thus, unexpected, risky events, normally modelled by the distribution's tail and mapped as the voltage socks that cause the state to vary aggressively, are not considered by the LEQG controller. Worth noticing that further improvements are not possible given that LEQG experiences the so-called \textit{neurotic breakdown} for $\theta<-0.05$. 
\par Probably one of the main drawbacks of LEQG is that values of $\theta$ that seem to be `stable' for some noise realizations, turn out to be `unstable' for others (especially for those with large shocks). In other words, in order to \textit{guarantee} for a stable value of $\theta$ over all noise realizations, we have to deteriorate the performance of the LEQG controller significantly. To put it differently, compared to our state-output risk-averse controller, LEQG is not of the `set and forget' type.
Lastly note that, on the one hand, state varies more aggressively only when the system experiences successive double or triple shocks. On the other, state experiences the internal shock with some (constant) time delay.
\subsubsection{Case 2. (Gaussian process noise-Skewed measurement noise)}

This case demonstrates the functionality of the second constraint. In particular, internal conditions in $S$ are assumed nominal, thus modelled with a Gaussian process noise, while the output disturbance simulates unexpected-risky events on the resistor voltage $y_{t}$. 
\begin{figure}[h!]
\centering
\includegraphics[scale=0.18]{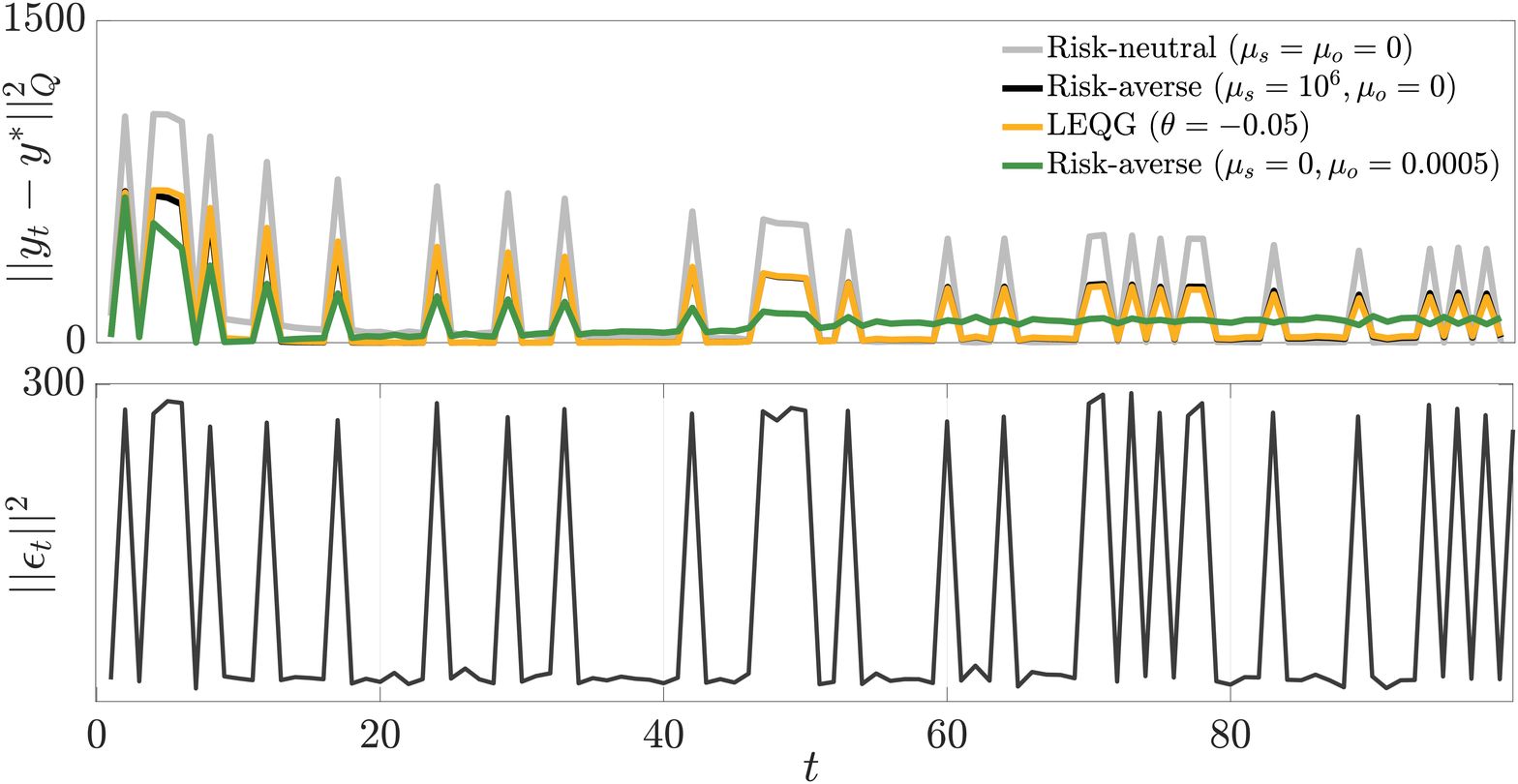}
\caption{Output variation w.r.t. the surrounding shocks $\epsilon_{t}$ for the various control policies (top). Output voltage shocks in one-to-one correspondence with the output variation (bottom).}
\label{c2}
\end{figure}
Figure \ref{c2} (top) shows how the energy on the
resistor $R_{2}$ alters with time. 
First and foremost the resistor's energy remains totally unguarded from the upcoming voltage shocks under the risk-neutral policy.
Note the similar performance between our risk-averse controller with $\mu_{s,o}=(10^{6},0)$, (black) and the LEQG (yellow), which without further improvement in between, breaks down for $\theta=-0.08$. The Gaussian process noise kills the affine terms of the risk-averse policy rendering our controller able to suppress risky events only by the state penalty matrix inflation, recasting its architecture similar to the one that, \textit{by default}, LEQG enjoys. However, our controller remains stable despite the extremes of $\mu_{s}$.
\par
On the other hand, by activating the second constraint (forest green) with $\mu_{s,o}=(0,5\times{10}^{-4})$ the controller counteracts risky events in a direct manner by rendering the resistor's energy inert to the surrounding shocks in the cost of sacrificing average performance. Lastly, note the instant, and more direct influence of the surrounding voltage disturbances (bottom) to the output. 
\subsubsection{Case 3. (Skewed process noise-Skewed measurement noise)}
Fig. \ref{SpSm} (top) shows how the various policies perform when skewed model prior combines with skewed output prior. In order to clearly depict the corresponding resistor's energy variations due to the internal, and output shocks, we separate the latter by choosing $(p_{1},p_{2})=(0.95,0.05)$ for both types of noises. Thus, the first peak (top) is due to $\epsilon_{t}$ (bottom) while the rest are due to the inner ones $w_{t}$ (middle). 
\begin{figure}[h!]
    \centering
    \includegraphics[scale=0.18]{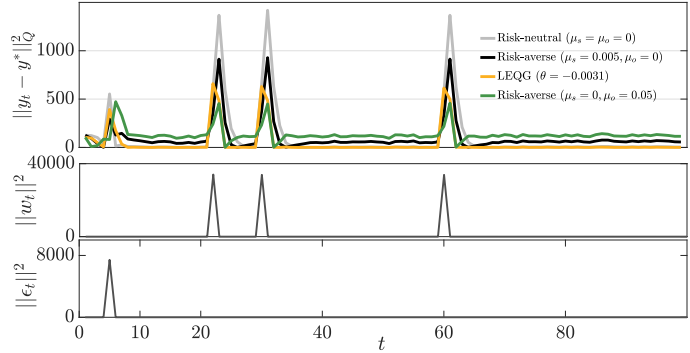}
    \caption{Output varriation due to internal and external disturbances (top). Interrnal shocks (middle), and output shocks (bottom). Note how each type of shock influences the output variation.}
    \label{SpSm}
\end{figure}
By construction, the first constraint (black) with $\mu_{s}=0.005$, reduces the influence of the  internal voltage shocks. On the contrary, the second constraint (forest green) with $\mu_{s,o}=(0,0.05)$ counteracts for the internal and external disturbances equally. Note that a larger value of $\mu_{o}$ can completely flatten the effect of the external disturbance in the cost of increasing effect of the internal disturbances. Lastly, LEQG  with $\theta=-0.0031$ achieves an adequate average performance while sufficiently reduces variability. However, note that the specific value for $\theta$ might be unstable for other noise realizations. Once again, simulations indicate that guaranteed stable values of $\theta$ (over all noise realizations) comes with significant performance deterioration.

\section{Conclusion and future work}
We proposed a new state-output risk-constrained optimal controller
for partially observed linear systems.  Our formulation considers  systems disrupted by arbitrary process and/or output noise (with finite fourth order moments) and the resulting policy admits closed form while it maintains favorable characteristics of the classical LQ strategy, i.e. it can be pre-computed off-line and stored.
By constraining the squared error predictive variability of the state and output penalties, the controller accounts for both internal shocks and also shields against external (i.e., output-induced) risky events. Our numerical results confirmed our theoretical analysis and corroborated the usefulness of state-output constrained control via a realistic indicative example. 
\par 
In addition to those referred in \cite{tsiamis2020risk},
interesting extensions of the proposed problem include but are not limited to risk-averse quadratic state regulation where the risk constraints consider restricting the volatility of the  distance between the control system and an adversarial target, e.g.,
\begin{align}
\mathbb{E}\left\{\sum_{t=1}^{N}\left[||x_{t}-g_{t}||^{2}_{Q_{s}} -\mathbb{E}\left( ||x_{t}-g_{t}||^{2}_{Q_{s}}\big|~ \overline{\mathscr{F}}_{t-1}\right)\right]^{2}\right\} \leq \epsilon,\nonumber
\end{align}
where $g_{t}$ denotes the state of an adversarial target. Such extensions are fruitful topics for further investigation.

\section{Appendix}

\subsection{{Proof of Proposition 1}}\label{proofprop1}

\subsubsection{State risk constraint re-formulation}
Similarly to \cite{tsiamis2020risk}, recall \textit{Assumption \ref{as1}}, and the causality constraint $u_{t}\in\mathcal{L}_{2}(\mathscr{F}_{t})$. We may write
\begin{align}
\tilde{x}_{t}= \mathbb{E}\{x_{t} |\overline{\mathscr{F}}_{t-1}\}&= \mathbb{E}\{Ax_{t-1}+Bu_{t-1}+w_{t} |\overline{\mathscr{F}}_{t-1}\}\nonumber\\[5pt]
&=Ax_{t-1}+Bu_{t-1}+\bar{w}\nonumber\\[5pt]
&=x_{t}-\delta_{t}, \label{tsoutses}
\end{align}
where we declared $\delta_{t}=w_{t}-\bar{w}$. The quadratic term may be written as 
\begin{align}
x^{\top}_{t}Q_{s}x_{t}=\tilde{x}^{\top}_{t}Q_{s}\tilde{x}_{t}+2\tilde{x}^{\top}Q_{s}\delta_{t}+\delta^{\top}_{t}Q_{s}\delta_{t}~,
\end{align}
and therefore
\begin{align}
\mathbb{E}\{x^{\top}_{t}Q_{s}x_{t}|\overline{\mathscr{F}}_{t-1}\}=\tilde{x}^{\top}_{t}Q_{s}\tilde{x}_{t}+\mathrm{tr}(Q_{s}W),
\end{align}
where we defined
\begin{align}
W:=\mathbb{E}\{(w_{t}-\bar{w})(w_{t}-\bar{w})^{\top} \}=\mathbb{E}\{\delta_{t}\delta^{\top}_{t}\}
\end{align}
as the process noise covariance. For the squared difference between the two we obtain 
\begin{align}
\Delta^{2}_{s,t}&=\Big[x^{\top}_{t}Q_{s}x_{t}-\mathbb{E}\big\{x^{\top}_{t}Q_{s}x_{t}|\overline{\mathscr{F}}_{t-1} \big\}\Big]^{2}\nonumber\\
&=\Big[2\tilde{x}^{\top}Q_{s}\delta_{t}+\delta^{\top}_{t}Q_{s}\delta_{t}-\mathrm{tr}(WQ_{s})\Big]^{2}\nonumber\\[5pt]
&=(\delta^{\top}_{t}Q_{s}\delta_{t}-\mathrm{tr}(WQ_{s}))^{2}+4\tilde{x}^{\top}_{t}Q_{s}\delta_{t}(\delta^{\top}_{t}Q\delta_{t}-\mathrm{tr}(WQ_{s}))+4\tilde{x}^{\top}Q_{s}\delta_{t}\delta^{\top}_{t}Q_{s}\tilde{x}_{t}.\label{difference}
\end{align}
At this point we define 
\begin{align}
    \mathrm{M}_{w}&\triangleq\mathbb{E}\Big\{(w_{t}-\bar{w})(w_{t}-\bar{w})^{\top}Q_{s}(w_{t}-\bar{w})\Big\}\nonumber,~\text{and}\\
    m_{w}&\triangleq \mathbb{E}\Big\{\Big((w_{t}-\bar{w})^{\top}Q_{s}(w_{t}-\bar{w})-\mathrm{tr}(Q_{s}W)\Big)^{2}\Big\}.
\end{align}
Thus, 
\begin{align}
\mathbb{E}\{\Delta^{2}_{s,t}\}&=m_{w}+\mathbb{E}\{4\tilde{x}^{\top}_{t}Q_{s}\delta_{t}(\delta^{\top}_{t}Q_{s}\delta_{t}-\mathrm{tr}(WQ_{s}))+4\tilde{x}^{\top}Q_{s}\delta_{t}\delta^{\top}_{t}Q_{s}\tilde{x}_{t} \}\nonumber\\[5pt]
   &= m_{w}+\mathbb{E}\{4\tilde{x}^{\top}_{t}Q_{s}\delta_{t}\delta_{t}Q\delta_{t}\}-\mathbb{E}\{4\tilde{x}^{\top}_{t}Q_{s}\delta_{t}\mathrm{tr}(Q_{s}W)\}+4\mathbb{E}\{\tilde{x}_{t}Q_{s}\delta_{t}\delta^{\top}_{t}Q_{s}\tilde{x}_{t}\}.
   \end{align}
By construction of $\tilde{x}_{t}$ and by first projecting onto $\overline{\mathscr{F}}_{t-1}$, for the first term we obtain
\begin{align}
   \mathbb{E}\{4\tilde{x}^{\top}_{t}Q_{s}\delta_{t}\delta_{t}Q_{s}\delta_{t}\}&=4\mathbb{E}\{ \mathbb{E}\{ \tilde{x}^{\top}_{t}Q_{s}\delta_{t}\delta_{t}Q_{s}\delta_{t}|\overline{\mathscr{F}}_{t-1}\}\}\nonumber\\[5pt]
   &=4\mathbb{E}\{\tilde{x}^{\top}Q_{s}\mathrm{M}_{w}\}
\end{align}
Regarding the second term, $\mathbb{E}\{\delta_{t}\}=0$ implies $\mathbb{E}\{4\tilde{x}^{\top}_{t}Q\delta_{t}\mathrm{tr}(Q_{s}W)\}=0$, while the third one reads 
\begin{align}
  4\mathbb{E}\{\tilde{x}^{\top}_{t}Q_{s}\delta_{t}\delta^{\top}_{t}Q_{s}\tilde{x}_{t}\}&=4\mathbb{E}\{\tilde{x}^{\top}_{t}Q_{s}WQ_{s}\tilde{x}_{t}\}.
\end{align}
Thus, 
\begin{align}
    \mathbb{E}\{\Delta^{2}_{s,t}\}&=m_{w}+\mathbb{E}\{4\tilde{x}^{\top}_{t}Q_{s}WQ_{s}\tilde{x}_{t}+4\tilde{x}_{t}Q_{s}\mathrm{M}_{w}\}\nonumber\\[5pt]
    &=m_{w}+\mathbb{E}\{4x_{t}^{\top}Q_{s}WQ_{s}x_{t}-8x^{\top}_{t}Q_{s}WQ_{s}\delta_{t}+4\delta^{\top}_{t}Q_{s}WQ_{s}\delta_{t}+4x_{t}Q_{s}\mathrm{M}_{w}-4\delta_{t}Q_{s}\mathrm{M}_{w}\}\nonumber\\[5pt]
    &=m_{w}+\mathbb{E}\{4x_{t}^{\top}Q_{s}WQ_{s}x_{t}+4x_{t}Q_{s}\mathrm{M}_{w}
    -8x^{\top}_{t}Q_{s}WQ_{s}\delta_{t}+4\delta^{\top}_{t}Q_{s}WQ_{s}\delta_{t}\}\nonumber\\[5pt]
    &=m_{w}+4(\mathrm{tr}(Q_{s}W))^{2}-8(\mathrm{tr}(Q_{s}W))^{2}+\mathbb{E}\{4x_{t}^{\top}Q_{s}WQ_{s}x_{t}+4x_{t}Q_{s}\mathrm{M}_{w}\}\nonumber\\[5pt]
    &=m_{w}-4(\mathrm{tr}(Q_{s}W))^{2}+\mathbb{E}\{4x_{t}^{\top}Q_{s}WQ_{s}x_{t}+4x_{t}Q_{s}\mathrm{M}_{w}\}
\end{align}
Therefore, the constraint recasts to 
\begin{align}
    \sum^{t=N}_{t=1}\mathbb{E}\{\Delta^{2}_{s,t}\}=\sum^{t=N}_{t=1}\mathbb{E}\{4x_{t}^{\top}Q_{s}WQ_{s}x_{t}+4x_{t}Q_{s}\mathbb{M}_{w}\}+Nm_{4}-4N(\mathrm{tr}(Q_{s}W))^{2}<\epsilon_{s}
\end{align}
or 
\begin{align}
\sum^{t=N}_{t=1}\mathbb{E}\{4x_{t}^{\top}Q_{s}WQ_{s}x_{t}+4x_{t}Q_{s}\mathrm{M}_{w}\}<\bar{\epsilon}_{s}
\end{align}
whith $\bar{\epsilon}_{s}=\epsilon_{s}-Nm_{w}+4N(\mathrm{tr}(Q_{s}W))^{2}$.

\subsubsection{Output risk constraint re-formulation}

Tore-formulate the second constraint, recall \textit{Assumption \ref{as1}}, the causality constraint $u_{t}\in\mathcal{L}_{2}(\FF_{t})$, and consider \eqref{out}, along with \eqref{tsoutses}. Therefore, we may write
$\tilde{y}_{t}=y_{t}-p_{t}$, with $p_{t}=C[w_{t}-\bar{w}]+\epsilon_{t}-\bar{\epsilon}$, and $\mathbb{E}\{p_{t}\}=0$, and subsequently write the quadratic as 
\begin{align}
    y^{\top}_{t}Q_{o} y_{t}=\tilde{y}^{\top}Q_{o} \tilde{y}_{t}+2\tilde{y}^{\top}_{t}Q_{o} p_{t}+p^{\top}_{t}Q_{o}p_{t}.
\end{align}
Hence, 
\begin{align}
    \mathbb{E}\{y^{\top}_{t}Q_{o}y_{t}|\overline{\mathscr{F}}_{t-1}\}&=\tilde{y}^{\top}Q_{o}\tilde{y}_{t}+\mathbb{E}\{p^{\top}_{t}Q_{o}p_{t}\}\nonumber\\[5pt]
    &=\tilde{y}^{\top}Q_{o}\tilde{y}_{t}+\mathrm{tr}(Q_{o}P).
\end{align}
where $P\triangleq\mathbb{E}\{p_{t}p^{\top}_{t}\}$. The squared difference reads
\begin{align}
\Delta^{2}_{o,t}&=\Big[y^{\top}_{t}Q_{o}y_{t}-\mathbb{E}\big\{y^{\top}_{t}Q_{o}y_{t}|\overline{\mathscr{F}}_{t-1} \big\}\Big]^{2}\nonumber\\
&=\Big[2\tilde{y}_{t}^{\top}Q_{o}p_{t}+p^{\top}_{t}Q_{o}p_{t}-\mathrm{tr}(Q_{o}P)\Big]^{2}\nonumber\\[5pt]
&=(p^{\top}_{t}Q_{o}p_{t}-\mathrm{tr}(Q_{o}P))^{2}\nonumber\\[5pt]&\hspace{10pt}+4\tilde{y}^{\top}_{t}Q_{o}p_{t}(p^{\top}_{t}Q_{o}p_{t}-\mathrm{tr}(Q_{o}P))\nonumber\\[5pt]
&\hspace{10pt}+4\tilde{y}_{t}^{\top}Q_{o}p_{t}p^{\top}_{t}Q_{o}\tilde{y}_{t},\label{difference22}
\end{align}
and by taking expectation in \eqref{difference22} we obtain
\begin{align}
   \hspace{-9pt} \mathbb{E}\big\{\Delta^{2}_{o,t}\big\}{=}m_{w\epsilon}+4\mathbb{E}\{\tilde{y}^{\top}_{t}Q_{o}\mathrm{M}\}+4\mathbb{E}\big\{\tilde{y}_{t}^{\top}Q_{o}PQ_{o}\tilde{y}_{t}\big\}.\label{katerinakigamiesai}
\end{align}
Further, $\tilde{y}_{t}=C\tilde{x}_{t}+\bar{\epsilon}=Cx_{t}-C\delta_{t}+\bar{\epsilon}$ and therefore, \eqref{katerinakigamiesai} may be written w.r.t. the system's state as
\begin{align}
    \mathbb{E}\big\{\Delta^{2}_{o,t}\big\}&=m_{w\epsilon}+\mathbb{E}\{4x_{t}^{\top}C^{\top}Q_{o}\mathrm{M}\}+4\bar{\epsilon}^{\top}Q_{o}\mathrm{M}
    \nonumber\\[5pt]&+4\mathbb{E}\big\{(x^{\top}_{t}C^{\top}+\zeta_{t}^{\top})Q_{o}PQ_{o}(Cx_{t}+\zeta_{t})\big\},\label{katerinakigamiesai2}
\end{align}
where we declared $\zeta_{t}=-C\delta_{t}+\bar{\epsilon}$. 
Thus, 
\begin{align}
\mathbb{E}\big\{\Delta^{2}_{o,t}\big\}&=m_{w\epsilon}+\mathbb{E}\big\{4x^{\top}_{t}C^{\top}Q_{o}\mathrm{M}\big\}\nonumber\\[5pt] &\hspace{10pt}+\mathbb{E}\big\{4x^{\top}_{t}C^{\top}Q_{o}PQ_{o}Cx_{t}\big\}\nonumber\\[5pt] &\hspace{10pt}+4\bar{\epsilon}Q_o\mathrm{M}-8\mathrm{tr}(C^{\top}Q_{o}PQ_{o}CW)\nonumber\\[5pt]
&\hspace{10pt}+\mathbb{E}\big\{8x^{\top}_{t}C^{\top}Q_{o}PQ_{o}\bar{\epsilon}\big\}
\nonumber\\[5pt]&\hspace{10pt}+\mathrm{tr}(4Q_{o}PQ_{o}Z),
\end{align}
where $ Z=\mathbb{E}\big\{\zeta_{t}\zeta_{t}^{\top}
    \big\}$. 
Therefore, the output risk constraint may be written as
\begin{align}
    \sum
    ^{N}_{t=1}\mathbb{E}\big\{4x^{\top}_{t}C^{\top}Q_{o}PQ_{o}Cx_{t}+4x^{\top}_{t}C^{\top}Q_{o}(\mathrm{M}+2P Q_{o} \bar{\epsilon})\big\}<\bar{\epsilon}_{o},\nonumber
\end{align}
where $\bar{\epsilon}_{o}{=}{-}Nm_{w\epsilon}{-}4N\bar{\epsilon}^{\top}Q_{o}\mathrm{M}_{}{-}8N\mathrm{tr}(C^{\top}Q_{o}PQ_{o}CW){-}4N\mathrm{tr}(Q_{o}PQ_{o}Z){+}\epsilon_{o}.$
\subsection{{Proof of Lemma 2}}\label{lemma2appendix}
The solution of \eqref{dynamics} may be written as 
\begin{align}
x_{t}=x^{r}_{t}+x^{d}_{t}
\end{align}
where $x^{r}_{t}$ transitions based on
\begin{align}
    x^{r}_{t+1}=Ax^{r}_{t}+w_{t}
\end{align}
with $x^{r}_{0}=x_{0}$, while $x^{d}_{t}$ according to
\begin{align}
    x^{d}_{t+1}=Ax^{d}_{t}+u_{t}
\end{align}
with $x^{d}_{0}=0$. Subsequently, \begin{align}
    y_{t}=y^{r}_{t}+y^{d}_{t},
\end{align}
where $y^{r}_{t}=Cx^{r}_{t}+v_{t}$, and $y^{d}_{t}=Cx^{d}_{t}$.
We will show that \begin{align}
    \mathscr{F}_{t}\triangleq\sigma(y_{0:t},u_{0:t-1})\equiv\sigma(y^{r}_{0:t}),
\end{align}
 By construction, $x^{d}_{t}$ and $y^{d}_{t}$ are deterministic functions of $u_{0:t}$. Since $y^{r}_{t}=y_{t}-y^{d}_{t}$, $y^{r}_{t}$ is a function of $(y_{t},u_{0:t})$. Thus, $y^{r}_{0:t}$ is a function of $(y_{0:t},u_{0:t})$ and as a result
$\sigma(y^{r}_{0:t})\subseteq\sigma(y_{0:t},u_{0:t-1})=\mathscr{F}_{t}$.\\
To prove the reverse inclusion, fix a (deterministic) policy 
$\pi$,
and further write $y_{0:t}=y^{r}_{0:t}+y^{d}_{0:t}$. Note that 
\begin{align}
y^{d}_{0}&=Cx_{0}\nonumber\\[5pt]
y_{1}^{d}&=C(Ax_{0}+B\pi(y^{r}_{0}+Cx_{0}))\nonumber\\[5pt]
&=\varphi_{1}(y^{r}_{0})\nonumber\\[5pt]
y^{d}_{2}&=C(A^{2}x_{0}+AB\pi(y^{r}_{0}+Cx_{0})+B\pi(y^{r}_{0:1}+y^{d}_{0:1}))\nonumber\\[5pt]
&=\varphi_{2}(y^{r}_{0:1})\nonumber\\[5pt]
&\vdots\nonumber\\[5pt]
y^{d}_{n}&=\varphi_{n}(y^{r}_{0:n-1}),\nonumber
\end{align}
where $\varphi_{n}:\mathbb{R}^{q\times{n}}\rightarrow \mathbb{R}^{q}$ are deterministic functions. Thus, 
\begin{align}
    u_{0:t-1}&=\pi(y_{0:t-1})
    =\pi(y^{r}_{0:t-1}+\varphi_{t}(y^{r}_{0:t-1})).\nonumber
\end{align}
As a result,
\begin{align}
(y_{0:t},u_{0:t-1}){=}\big(y^{r}_{0:t}+\varphi_{t}(y^{r}_{0:t-1}),\pi(y^{r}_{0:t-1}+\varphi_{t}(y^{r}_{0:t-1}))\big)\nonumber
\end{align}
and therefore
\begin{align}
 \sigma(y_{0:t},u_{0:t-1})\subseteq\sigma(y^{r}_{0:t})\nonumber
\end{align}
which concludes the proof.

\subsection{{Proof of Theorem 2}}\label{th2appendix}
\subsubsection{Bellman's Equation}
Consider the tail sub-problem
\begin{align}
\mathcal{L}^{*}(\mathscr{F}_{k};\mu_{s,o})&\triangleq\inf_{u_{k:N-1}}\mathbb{E}\Bigg\{ g(x_{N};\mu_{s,o})+\sum^{N-1}_{t=k}g(x_{t},u_{t};\mu_{s,o})\Bigg|\mathscr{F}_{k}\Bigg\}\nonumber\\[5pt]
&=\inf_{u_{k:N-1}}\mathbb{E}\Bigg\{ g(x_{N};\mu_{s,o})+\sum^{N-1}_{t=k+1}g(x_{t},u_{t};\mu_{s,o})+g(x_{k},u_{t};\mu_{s,o})\Bigg|\mathscr{F}_{k}\Bigg\}\nonumber\\[5pt]
&=\inf_{u_{k:N-1}}\mathbb{E}\Bigg\{ g(x_{N};\mu_{s,o})+\sum^{N-1}_{t=k+1}g(x_{t},u_{t};\mu_{s,o})+g(\widehat{x}_{k|k},u_{k};\mu_{s,o})+\phi(\widehat{x}_{k|k},e_{k};\mu_{s,o})\Bigg|\mathscr{F}_{k}\Bigg\}\nonumber\\[5pt]
&=\inf_{u_{k:N-1}}\mathbb{E}\Bigg\{ g(x_{N};\mu_{s,o})+\sum^{N-1}_{t=k+1}g(x_{t},u_{t};\mu_{s,o})+g(\widehat{x}_{k|k},u_{k};\mu_{s,o})+\mathrm{tr}(Q_{\mu_{s,o}}\Sigma_{k})\Bigg|\mathscr{F}_{k}\Bigg\}.\label{papoulias}
\end{align}
The latter rests on the fact that
$\mathbb{E}\{\phi(\widehat{x}_{k|k},e_{k};\mu_{s,o}) |\mathscr{F}_{k}\}=\mathrm{tr}(Q_{\mu_{s,o}}\mathbb{E}\{e_{k}e^{\top}_{k}|\mathscr{F}_{k}\})$. Thus
\begin{align}
\mathcal{L}^{*}(\mathscr{F}_{k};\mu_{s,o})&=\inf_{u_{k}}~\mathbb{E}\Bigg\{\inf_{u_{k+1:N-1}}\Bigg(\mathbb{E}\Bigg\{ g(x_{N};\mu_{s,o})+\sum^{N-1}_{t=k+1}g(x_{t},u_{t};\mu_{s,o})\Bigg|\mathscr{F}_{k+1}\Bigg\}\Bigg)\nonumber\\[5pt]
&\hspace{4cm}+g(\widehat{x}_{k|k},u_{k};\mu_{s,o})+\mathrm{tr}(Q_{\mu_{s,o}}\mathbb{E}\{e_{k}e^{\top}_{k}|\mathscr{F}_{k}\})\Bigg|\mathscr{F}_{k}\Bigg\}\nonumber\\[5pt]
&=\inf_{u_{k}}~\mathbb{E}\Bigg\{\mathcal{L}^{*}(\mathscr{F}_{k+1};\mu_{s,o})+g(\widehat{x}_{k|k},u_{k};\mu_{s,o})+\mathrm{tr}(Q_{\mu_{s,o}}\mathbb{E}\{e_{k}e^{\top}_{k}|\mathscr{F}_{k}\})\Bigg|\mathscr{F}_{k}\Bigg\}\nonumber\\[5pt]
&=\inf_{u_{k}}~\Big(\mathbb{E}\Big\{\mathcal{L}^{*}(\mathscr{F}_{k+1};\mu_{s,o})\Big|\mathscr{F}_{k}\Big\}+g(\widehat{x}_{k|k},u_{k};\mu_{s,o})+\mathrm{tr}(Q_{\mu_{s,o}}\mathbb{E}\{e_{k}e^{\top}_{k}|\mathscr{F}_{k}\})\Big)\nonumber\\[5pt]
&=\inf_{u_{k}}~\Big(\mathbb{E}\Big\{\mathcal{L}^{*}(\mathscr{F}_{k+1};\mu_{s,o})\Big|\mathscr{F}_{k}\Big\}+g(\widehat{x}_{k|k},u_{k};{\mu_{s,o}})\Big)+\mathrm{tr}(Q_{\mu_{s,o}}\mathbb{E}\{e_{k}e^{\top}_{k}|\mathscr{F}_{k}\}),
\end{align}
where the latter is based on \textit{Lemma \ref{lem2}}. Thus, Bellman equation reads 
\begin{align}
\mathcal{L}^{*}(\mathscr{F}_{k};\mu_{s,o})&=\inf_{u_{k}}~\Big(\mathbb{E}\Big\{\mathcal{L}^{*}(\mathscr{F}_{k+1};\mu_{s,o})\Big|\mathscr{F}_{k}\Big\}+g(\widehat{x}_{k|k},u_{k};{\mu_{s,o}})\Big)+\mathrm{tr}(Q_{\mu_{s,o}}\Sigma_{k}).
\end{align}

\subsubsection{Optimal Policy}
The proof follows by induction. We demonstrate the proof for the root. Bellman's starts from $k=N$ with 
\begin{align}
    \mathcal{L}^{*}(\mathscr{F}_{N};\mu_{s,o})&=\mathbb{E}\Big\{g(x_{N};\mu_{s,o})~\big|~\mathscr{F}_{N}\Big\}
    =\mathbb{E}\Big\{
    x^{\top}_{N}Q_{\mu_{s,o}}x_{N}+x^{\top}_{N}\mathrm{M}_{\mu_{s,o}}~\big|~\mathscr{F}_{N}\Big\}\nonumber
\end{align}
At this point decompose $x_{N}=\widehat{x}_{N|N}+e_{N}$ to obtain 
\begin{align}
    \mathcal{L}^{*}(\mathscr{F}_{N};\mu_{s,o})&=\mathbb{E}\Big\{
    \widehat{x}_{N|N}^{\top}Q_{\mu_{s,o}}\widehat{x}_{N|N}+\widehat{x}_{N|N}^{\top}\mathrm{M}_{\mu_{s,o}}+2\widehat{x}^{\top}_{N|N}Q_{\mu_{s,o}}e_{N}+e^{\top}_{N}Q_{\mu_{s,o}}e_{N}+e^{\top}_{N}\mathrm{M}_{\mu_{s,o}}
    ~\big|~\mathscr{F}_{N}\Big\}\nonumber\\[5pt]
    &=\widehat{x}_{N|N}^{\top}Q_{\mu_{s,o}}\widehat{x}_{N|N}+\widehat{x}_{N|N}^{\top}\mathrm{M}_{\mu_{s,o}}+\mathrm{tr}\Big( {Q}_{\mu_{s,o}}\Sigma_{N|N}\Big)
\end{align}
The latter rests of the fact that $\mathbb{E}\{ \widehat{x}_{N|N}^{\top}Q_{\mu_{s,o}}e_{N}|\mathscr{F}_{N}\}= \widehat{x}_{N|N}^{\top}Q_{\mu_{s,o}}\mathbb{E}\{e_{N}|\mathscr{F}_{N}\}=0$ and $\mathbb{E}\{e^{\top}_{N}\mathrm{M}_{\mu_{s,o}}|\mathscr{F}_{N}\}=
\mathbb{E}\{e^{\top}_{N}|\mathscr{F}_{N}\}\mathrm{M}_{\mu_{s,o}}=0$.
Thus, 
\begin{align}
\mathcal{L}^{*}(\mathscr{F}_{N};\mu_{s,o})&=\widehat{x}_{N|N}^{\top}Q_{\mu_{s,o}}\widehat{x}_{N|N}+\widehat{x}_{N|N}^{\top}\mathrm{M}_{\mu_{s,o}}+\mathrm{tr}\Big(Q_{\mu_{s,o}}\Sigma_{N|N}\Big)
\end{align}
or 
\begin{align}
\mathcal{L}^{*}(\mathscr{F}_{N};\mu_{s,o})&=\widehat{x}_{N|N}^{\top}V_{N}\widehat{x}_{N|N}+2\widehat{x}^{\top}_{N|N}T_{N}\bar{w}+\widehat{x}_{N|N}^{\top}S_{N}\mathrm{M}_{\mu_{s,o}}+p_{N},
\end{align}
where $V_{N}=Q_{\mu_{s,o}}$, $S_{N}=I$, $T_{N}=0$, and $p_{N}=\mathrm{tr}\Big(Q_{\mu_{s,o}}\Sigma_{N|N}\Big)$. Further, 
\begin{align}
\mathcal{L}^{*}(\mathscr{F}_{N-1};\mu_{s,o})&=\inf_{u_{N-1}}~\Big(\mathbb{E}\Big\{\mathcal{L}^{*}(\mathscr{F}_{N};\mu_{s,o})\Big|\mathscr{F}_{N-1}\Big\}+g(\widehat{x}_{N-1|N-1},u_{N-1};{\mu_{s,o}})\Big)+\mathrm{tr}(Q_{\mu_{s,o}}\Sigma_{N-1|N-1}).\label{63}
\end{align}
Starting with $\mathbb{E}\Big\{\mathcal{L}^{*}(\mathscr{F}_{N};\mu_{s,o})\Big|\mathscr{F}_{N-1}\Big\}$, we obtain

\begin{align}
\mathbb{E}\Big\{\mathcal{L}^{*}(\mathscr{F}_{N};\mu_{s,o})\Big|\mathscr{F}_{N-1}\Big\}&= \mathbb{E}\Big\{
\widehat{x}_{N|N}^{\top}V_{N}\widehat{x}_{N|N}+\widehat{x}_{N|N}^{\top}S_{N}\mathrm{M}_{\mu_{s,o}}+p_{N}\Big|\mathscr{F}_{N-1}\Big\},
\end{align}
and after decomposing $\widehat{x}_{N|N}=x_{N}-e_{N}$ we get
\begin{align}
\mathbb{E}\Big\{\mathcal{L}^{*}(\mathscr{F}_{N};\mu_{s,o})\Big|\mathscr{F}_{N-1}\Big\}&= \mathbb{E}\Big\{
{x}_{N}^{\top}V_{N}{x}_{N}+{x}_{N}^{\top}S_{N}\mathrm{M}_{\mu_{s,o}}+p_{N}-2x^{\top}_{N}V_{N}e_{N}+e^{\top}_{N}V_{N}e_{N}-e^{\top}_{N}S_{N}\mathrm{M}_{\mu_{s,o}}
\Big|\mathscr{F}_{N-1}\Big\}\nonumber\\[5pt]
&=\mathbb{E}\Big\{
{x}_{N}^{\top}V_{N}{x}_{N}+{x}_{N}^{\top}S_{N}\mathrm{M}_{\mu_{s,o}}+p_{N}\Big|\mathscr{F}_{N-1}\Big\}-2\mathrm{tr}(V_{N}H_{N|N-1})+\mathrm{tr}(V_{N}\Sigma_{N|N}).\label{65}
\end{align}
Further, given that $x_{N}=Ax_{N-1}+Bu_{N-1}+w_{N}$ the first term reads:
\begin{align}
&\mathbb{E}\Big\{
(Ax_{N-1}+Bu_{N-1}+w_{N})^{\top}V_{N}(Ax_{N-1}+Bu_{N-1}+w_{N})+(Ax_{N-1}+Bu_{N-1}+w_{N})^{\top}S_{N}\mathrm{M}(\mu_{s,o}+p_{N}\Big|\mathscr{F}_{N-1}\Big\}=\nonumber\\[5pt]
&=\mathbb{E}\Big\{
x^{\top}_{N-1}A^{\top}V_{N}Ax_{N-1}+2x^{\top}_{N-1}A^{\top}V_{N}Bu_{N-1}+2x^{\top}_{N-1}A^{\top}V_{N}w_{N}+u^{\top}_{N-1}B^{\top}V_{N}Bu_{N-1}+2u^{\top}_{N-1}B^{\top}V_{N}w_{N}\nonumber\\[5pt]
&+w^{\top}_{N}V_{N}w_{N}+x^{\top}_{N-1}A^{\top}S_{N}\mathrm{M}_{\mu_{s,o}}+u^{\top}_{N-1}B^{\top}S_{N}\mathrm{M}_{\mu_{s,o}}+w^{\top}_{N}S_{N}\mathrm{M}_{\mu_{s,o}}
\Big|\mathscr{F}_{N-1}\Big\},\label{66}
\end{align}
and after decomposing $x_{N-1}=\widehat{x}_{N-1|N-1}+e_{N-1}$, the first term in \eqref{65} yields
\begin{align}
&\mathbb{E}\Big\{
{x}_{N}^{\top}V_{N}{x}_{N}+{x}_{N}^{\top}S_{N}\mathrm{M}_{\mu_{s,o}}+p_{N}\Big|\mathscr{F}_{N-1}\Big\}=\nonumber\\[5pt]
&=\widehat{x}^{\top}_{N-1|N-1}A^{\top}V_{N}A\widehat{x}_{N-1|N-1}+2\widehat{x}^{\top}_{N-1|N-1}A^{\top}V_{N}Bu_{N-1}+2\widehat{x}^{\top}_{N-1|N-1}A^{\top}V_{N}\bar{w}+u^{\top}_{N-1}B^{\top}V_{N}Bu_{N-1}\nonumber\\[5pt]
&+2u^{\top}_{N-1}B^{\top}V_{N}\bar{w}+\bar{w}^{\top}V_{N}\bar{w}+\mathrm{tr}(V_{N}W)+\widehat{x}^{\top}_{N-1|N-1}A^{\top}S_{N}\mathrm{M}_{\mu_{s,o}}+u^{\top}_{N-1}B^{\top}S_{N}\mathrm{M}_{\mu_{s,o}}+\bar{w}^{\top}S_{N}\mathrm{M}_{\mu_{s,o}}\nonumber\\[5pt]
&+\mathrm{tr}(A^{\top}V_{N}A\Sigma_{N-1|N-1})
\end{align}
Thus, 

\begin{align}
&\mathbb{E}\Big\{\mathcal{L}^{*}(\mathscr{F}_{N};\mu_{s,o})\Big|\mathscr{F}_{N-1}\Big\}\nonumber\\[5pt]
&=\widehat{x}^{\top}_{N-1|N-1}A^{\top}V_{N}A\widehat{x}_{N-1|N-1}+2\widehat{x}^{\top}_{N-1|N-1}A^{\top}V_{N}Bu_{N-1}+2\widehat{x}^{\top}_{N-1|N-1}A^{\top}V_{N}\bar{w}+u^{\top}_{N-1}B^{\top}V_{N}Bu_{N-1}\nonumber\\[5pt]
&+2u^{\top}_{N-1}B^{\top}V_{N}\bar{w}+\bar{w}^{\top}V_{N}\bar{w}+\mathrm{tr}(V_{N}W)+\widehat{x}^{\top}_{N-1|N-1}A^{\top}S_{N}\mathrm{M}_{\mu_{s,o}}+u^{\top}_{N-1}B^{\top}S_{N}\mathrm{M}_{\mu_{s,o}}+\bar{w}^{\top}S_{N}\mathrm{M}_{\mu_{s,o}}\nonumber\\[5pt]
&+\mathrm{tr}(A^{\top}V_{N}A\Sigma_{N-1|N-1})-2\mathrm{tr}(V_{N}H_{N|N-1})+\mathrm{tr}(V_{N}\Sigma_{N|N}).
\end{align}
As a result, the minimization of
\begin{align}
\inf_{u_{N-1}}~\Big(\mathbb{E}\Big\{\mathcal{L}^{*}(\mathscr{F}_{N};\mu_{s,o})\Big|\mathscr{F}_{N-1}\Big\}+g(\widehat{x}_{N-1|N-1},u_{N-1};{\mu_{s,o}})\Big)
\end{align}
in \eqref{63} rests on \textit{Lemma} \ref{lem2},  and provides the optimal control \begin{align}
 u^{*}_{N-1}=&-\left(B^{\top} V_{N} B+R\right)^{-1}B^{\top}V_{N}A\widehat{x}_{N-1|N-1}\nonumber\\[5pt]&-\left(B^{\top} V_{N} B+R\right)^{-1}B^{\top}V_{N}\bar{w}\nonumber\\[5pt]&-\frac{1}{2}\left(B^{\top} V_{N} B+R\right)^{-1}B^{\top}S_{N}\mathrm{M}_{\mu_{s,o}}
\end{align}
or 
\begin{align}
u^{*}_{N-1}=K_{N-1}\widehat{x}_{N-1|N-1}+h_{N-1}+l_{N-1}  
\end{align}
where 
\begin{align}
    K_{N-1}&=-\left(B^{\top} V_{N} B+R\right)^{-1}B^{\top}V_{N}A\\
    h_{N-1}&=-\left(B^{\top} V_{N} B+R\right)^{-1}B^{\top}V_{N}\bar{w}\\
    l_{N-1}&=-\frac{1}{2}\left(B^{\top} V_{N} B+R\right)^{-1}B^{\top}S_{N}\mathrm{M}_{\mu_{s,o}}
\end{align}
By substituting $u^{*}_{N-1}$, we obtain the optimal cost-to-go at $t=N-1$ as follows
\begin{align}
&\widehat{x}^{\top}_{N-1|N-1}\big(A^{\top}V_{N}+Q_{\mu_{s,o}}-A^{\top}V_{N}B(B^{\top}V_{N}B+R)^{-1}B^{\top}V_{N}A
\big)\widehat{x}_{N-1|N-1}\nonumber\\[5pt]
&+\widehat{x}^{\top}_{N-1|N-1}\big(2A^{\top}-2(B^{\top}V_{N}B+R)^{-1}B^{\top}V_{N}AB^{\top}V_{N}
\big)\bar{w}\nonumber\\[5pt]
&+\widehat{x}^{\top}_{N-1|N-1}\big(A^{\top}S_{N}-A^{\top}V_{N}B(B^{\top}V_{N}B+R)^{-1}B^{\top}S_{N}+I
\big)\mathrm{M}_{\mu_{s,o}}\nonumber\\[5pt]
&+(h_{N-1}+l_{N-1})^{\top}(B^{\top}V_{N}B+R)(h_{N-1}+l_{N-1})+(h_{N-1}+l_{N-1})^{\top}\big(2B^{\top}V_{N}\bar{w}+B^{\top}S_{N}\mathrm{M}_{\mu_{s,o}}
\big)\nonumber\\[5pt]
&+\bar{w}^{\top}V_{N}\bar{w}+\mathrm{tr}(V_{N}W)+\bar{w}^{\top}S_{N}\mathrm{M}_{\mu_{s,o}}+\mathrm{tr}(V_{N}\Sigma_{N|N-1})-2\mathrm{tr}(V_{N}H_{N|N-1})+\mathrm{tr}(V_{N}\Sigma_{N-1|N-1}).
\end{align}
The second line reads 
\begin{align}
&\widehat{x}^{\top}_{N-1|N-1}\big(2A^{\top}-2(B^{\top}V_{N}B+R)^{-1}B^{\top}V_{N}AB^{\top}V_{N}
\big)\bar{w}\nonumber\\[5pt]
&+\widehat{x}^{\top}_{N-1|N-1}\big(A^{\top}S_{N}-A^{\top}V_{N}B(B^{\top}V_{N}B+R)^{-1}B^{\top}S_{N}
\big)\mathrm{M}_{\mu_{s,o}}\nonumber\\[5pt]
&=2\widehat{x}^{\top}_{N-1|N-1}(A^{\top}V_{N}+K^{\top}_{N-1}B^{\top}V_{N}
)\bar{w}\nonumber\\[5pt]
&=2\widehat{x}^{\top}_{N-1|N-1}\big(A^{\top}+K^{\top}_{N-1}B^{\top}
\big)V_{N}\bar{w}\nonumber\\[5pt]
&=2\bar{w}^{\top}V_{N}\big( A+BK_{N-1}\big)\widehat{x}_{N-1|N-1}
\nonumber\\[5pt]
&=2\widehat{x}^{\top}_{N-1|N-1}\big(A+BK_{N-1}\big)^{\top}V_{N}\bar{w}.
\end{align}
The third line reads 
\begin{align}
&\widehat{x}^{\top}_{N-1|N-1}\big(A^{\top}S_{N}-A^{\top}V_{N}BB^{\top}V_{N}B+R)^{-1}B^{\top}S_{N}
+I\big)\mathrm{M}_{\mu_{s,o}}\nonumber\\[5pt]
&=\widehat{x}^{\top}_{N-1|N-1}\big((A^{\top}-K^{\top}_{N-1}B^{\top})S_{N}+I
\big)\mathrm{M}_{\mu_{s,o}}.\nonumber\\[5pt]
\end{align}
For the fourth line 
\begin{align}
(h_{N-1}+l_{N-1})^{\top}(B^{\top}V_{N}B+R)(h_{N-1}+l_{N-1})+(h_{N-1}+l_{N-1})^{\top}\big(2B^{\top}V_{N}\bar{w}+B^{\top}S_{N}\mathrm{M}_{\mu_{s,o}},
\big),\nonumber\\[5pt]
\end{align}
just multiply and divide by $(B^{\top}V_{N}B+R)$ to obtain 
\begin{align}
&(h_{N-1}+l_{N-1})^{\top}(B^{\top}V_{N}B+R)(h_{N-1}+l_{N-1})\nonumber\\[5pt]
&+(h_{N-1}+l_{N-1})^{\top}(B^{\top}V_{N}B+R)(B^{\top}V_{N}B+R)^{-1}\big(2B^{\top}V_{N}\bar{w}+B^{\top}S_{N}\mathrm{M}_{\mu_{s,o}}
\big)\nonumber\\[5pt]
&=(h_{N-1}+l_{N-1})^{\top}(B^{\top}V_{N}B+R)(h_{N-1}+l_{N-1})\nonumber\\[5pt]
&+(h_{N-1}+l_{N-1})^{\top}(B^{\top}V_{N}B+R)(B^{\top}V_{N}B+R)^{-1}\big(2B^{\top}V_{N}\bar{w}+B^{\top}S_{N}\mathrm{M}_{\mu_{s,o}}
\big)\nonumber\\[5pt]
&=(h_{N-1}+l_{N-1})^{\top}(B^{\top}V_{N}B+R)(h_{N-1}+l_{N-1})\nonumber\\[5pt]
&-2(h_{N-1}+l_{N-1})^{\top}(B^{\top}V_{N}B+R)(B^{\top}V_{N}B+R)^{-1}\big(-B^{\top}V_{N}\bar{w}-\frac{1}{2}B^{\top}S_{N}\mathrm{M}_{\mu_{s,o}}
\big)\nonumber\\[5pt]
&=(h_{N-1}+l_{N-1})^{\top}(B^{\top}V_{N}B+R)(h_{N-1}+l_{N-1})\nonumber\\[5pt]
&-2(h_{N-1}+l_{N-1})^{\top}(B^{\top}V_{N}B+R)(h_{N-1}+l_{N-1}
\big)\nonumber\\[5pt]
&=-(h_{N-1}+l_{N-1})^{\top}(B^{\top}V_{N}B+R)(h_{N-1}+l_{N-1}
\big).
\end{align}
Thus, we may set 
\begin{align}
&V_{N-1}=A^{\top}V_{N}+Q_{\mu_{s,o}}-A^{\top}V_{N}B(B^{\top}V_{N}B+R)^{-1}B^{\top}V_{N}A,\nonumber\\[5pt]
&T_{N-1}=(A+BK_{N-1})^{\top}V_{N},\nonumber\\[5pt]
&S_{N-1}=(A+BK_{N-1})^{\top}S_{N}+I,\nonumber\\[5pt]
&p_{N-1}=p_{N}+\bar{w}^{\top}V_{N}\bar{w}+\mathrm{tr}(V_{N}W)+\bar{w}^{\top}S_{N}\mathrm{M}_{\mu_{s,o}}+\mathrm{tr}(V_{N}\Sigma_{N|N-1})-2\mathrm{tr}(V_{N}H_{N|N-1})+\mathrm{tr}(V_{N}\Sigma_{N-1|N-1})\nonumber\\[5pt]
&-(h_{N-1}+l_{N-1})^{\top}(B^{\top}V_{N}B+R)(h_{N-1}+l_{N-1}
\big),
\end{align}
and write the optimal cost-to-go at $t=N-1$ as
\begin{align}
\mathcal{L}^{*}(\mathscr{F}_{N-1};\mu_{s,o}) =\widehat{x}^{\top}_{N-1|N-1}V_{N-1}\widehat{x}_{N-1|N-1}+2\widehat{x}^{\top}_{N-1|N-1}T_{N-1}\bar{w}+\widehat{x}^{\top}_{N-1|N-1}S_{N-1}\mathrm{M}_{\mu_{s,o}}+p_{N-1}.
\end{align}
to conclude the proof for the root. The transition of truthfulness follows from similar steps.

\bibliographystyle{IEEEtran}
\bibliography{IEEEabrv,refs.bib,additional.bib}

\begin{thebibliography}{10}
\providecommand{\url}[1]{#1}
\csname url@samestyle\endcsname
\providecommand{\newblock}{\relax}
\providecommand{\bibinfo}[2]{#2}
\providecommand{\BIBentrySTDinterwordspacing}{\spaceskip=0pt\relax}
\providecommand{\BIBentryALTinterwordstretchfactor}{4}
\providecommand{\BIBentryALTinterwordspacing}{\spaceskip=\fontdimen2\font plus
\BIBentryALTinterwordstretchfactor\fontdimen3\font minus
  \fontdimen4\font\relax}
\providecommand{\BIBforeignlanguage}[2]{{%
\expandafter\ifx\csname l@#1\endcsname\relax
\typeout{** WARNING: IEEEtran.bst: No hyphenation pattern has been}%
\typeout{** loaded for the language `#1'. Using the pattern for}%
\typeout{** the default language instead.}%
\else
\language=\csname l@#1\endcsname
\fi
#2}}
\providecommand{\BIBdecl}{\relax}
\BIBdecl

\bibitem{kim2019bi}
S.-K. Kim, R.~Thakker, and A.-A. Agha-Mohammadi, ``Bi-directional value
  learning for risk-aware planning under uncertainty,'' \emph{IEEE Robotics and
  Automation Letters}, vol.~4, no.~3, pp. 2493--2500, 2019.

\bibitem{pereira2013risk}
A.~A. Pereira, J.~Binney, G.~A. Hollinger, and G.~S. Sukhatme, ``Risk-aware
  path planning for autonomous underwater vehicles using predictive ocean
  models,'' \emph{Journal of Field Robotics}, vol.~30, no.~5, pp. 741--762,
  2013.

\bibitem{ma2018risk}
W.-J. Ma, C.~Oh, Y.~Liu, D.~Dentcheva, and M.~M. Zavlanos, ``Risk-averse access
  point selection in wireless communication networks,'' \emph{IEEE Transactions
  on Control of Network Systems}, vol.~6, no.~1, pp. 24--36, 2018.

\bibitem{bennis2018ultrareliable}
M.~Bennis, M.~Debbah, and H.~V. Poor, ``Ultrareliable and low-latency wireless
  communication: Tail, risk, and scale,'' \emph{Proceedings of the IEEE}, vol.
  106, no.~10, pp. 1834--1853, 2018.

\bibitem{iwaki2021multi}
T.~Iwaki, J.~Wu, Y.~Wu, H.~Sandberg, and K.~H. Johansson, ``Multi-hop sensor
  network scheduling for optimal remote estimation,'' \emph{Automatica}, vol.
  127, p. 109498, 2021.

\bibitem{fan2020bayesian}
D.~D. Fan, J.~Nguyen, R.~Thakker, N.~Alatur, A.-a. Agha-mohammadi, and E.~A.
  Theodorou, ``Bayesian learning-based adaptive control for safety critical
  systems,'' in \emph{2020 IEEE international conference on robotics and
  automation (ICRA)}.\hskip 1em plus 0.5em minus 0.4em\relax IEEE, 2020, pp.
  4093--4099.

\bibitem{lindemann2021stl}
L.~Lindemann, N.~Matni, and G.~J. Pappas, ``{STL} {R}obustness {R}isk over
  {D}iscrete-{T}ime {S}tochastic {P}rocesses,'' \emph{arXiv preprint
  arXiv:2104.01503}, 2021.

\bibitem{Cardoso2019}
A.~R. Cardoso and H.~Xu, ``Risk-{A}verse {S}tochastic {C}onvex {B}andit,'' in
  \emph{International Conference on Artificial Intelligence and Statistics},
  vol.~89, Apr. 2019, pp. 39--47.

\bibitem{dentcheva2017statistical}
D.~Dentcheva, S.~Penev, and A.~Ruszczy{\'n}ski, ``Statistical estimation of
  composite risk functionals and risk optimization problems,'' \emph{Annals of
  the Institute of Statistical Mathematics}, vol.~69, no.~4, pp. 737--760,
  2017.

\bibitem{tsiamis2021linear}
A.~Tsiamis, D.~S. Kalogerias, A.~Ribeiro, and G.~J. Pappas, ``Linear
  {Q}uadratic {C}ontrol with {R}isk {C}onstraints,'' \emph{arXiv preprint
  arXiv:2112.07564}, 2021.

\bibitem{whittle}
P.~Whittle, ``Risk-{S}ensitive {L}inear/{Q}uadratic/{G}aussian {C}ontrol,''
  \emph{Advances in Applied Probability}, vol.~13, no.~4, pp. 764--777, 1981.

\bibitem{anderson2012optimal}
B.~D. Anderson and J.~B. Moore, \emph{Optimal filtering}.\hskip 1em plus 0.5em
  minus 0.4em\relax Courier Corporation, 2012.

\bibitem{A.2018}
L.~A. Prashanth and M.~Fu, ``Risk-{S}ensitive {R}einforcement {L}earning: {A}
  {C}onstrained {O}ptimization {V}iewpoint,'' \emph{arXiv preprint,
  arXiv:1810.09126}, Oct. 2018.

\bibitem{W.Huang2017}
W.~Huang and W.~B. Haskell, ``Risk-{A}ware {Q}-learning for {M}arkov {D}ecision
  {P}rocesses,'' in \emph{2017 IEEE 56th Annual Conference on Decision and
  Control, CDC 2017}, vol. 2018-Janua.\hskip 1em plus 0.5em minus 0.4em\relax
  IEEE, Dec. 2018, pp. 4928--4933.

\bibitem{Jiang2017}
D.~R. Jiang and W.~B. Powell, ``Risk-{A}verse {A}pproximate {D}ynamic
  {P}rogramming with {Q}uantile-{B}ased {R}isk {M}easures,'' \emph{Mathematics
  of Operations Research}, vol.~43, no.~2, pp. 554--579, Nov. 2018.

\bibitem{Kalogerias2018b}
D.~S. Kalogerias and W.~B. Powell, ``Recursive {O}ptimization of {C}onvex
  {R}isk {M}easures: {M}ean-{S}emideviation {M}odels,'' \emph{arXiv preprint,
  arXiv:1804.00636}, Apr. 2018.

\bibitem{Tamar2017}
A.~Tamar, Y.~Chow, M.~Ghavamzadeh, and S.~Mannor, ``Sequential {D}ecision
  {M}aking with {C}oherent {R}isk,'' \emph{IEEE Transactions on Automatic
  Control}, vol.~62, no.~7, pp. 3323--3338, Jul. 2017.

\bibitem{Vitt2018}
C.~A. Vitt, D.~Dentcheva, and H.~Xiong, ``Risk-{A}verse {C}lassification,''
  \emph{Annals of Operations Research}, Aug. 2019.

\bibitem{Zhou2018}
L.~Zhou and P.~Tokekar, ``An {A}pproximation {A}lgorithm for {R}isk-averse
  {S}ubmodular {O}ptimization,'' \emph{arXiv preprint, arXiv:1807.09358}, Jul.
  2018.

\bibitem{Ruszczynski2010}
A.~Ruszczy{\'{n}}ski, ``Risk-{A}verse {D}ynamic {P}rogramming for {M}arkov
  {D}ecision {P}rocesses,'' \emph{Mathematical Programming}, vol. 125, no.~2,
  pp. 235--261, Oct. 2010.

\bibitem{SOPASAKIS2019281}
P.~Sopasakis, D.~Herceg, A.~Bemporad, and P.~Patrinos, ``Risk-averse model
  predictive control,'' \emph{Automatica}, vol. 100, pp. 281 -- 288, 2019.

\bibitem{chapman2019cvar}
M.~P. {Chapman}, J.~{Lacotte}, A.~{Tamar}, D.~{Lee}, K.~M. {Smith}, V.~{Cheng},
  J.~F. {Fisac}, S.~{Jha}, M.~{Pavone}, and C.~J. {Tomlin}, ``A
  {R}isk-{S}ensitive {F}inite-{T}ime {R}eachability {A}pproach for {S}afety of
  {S}tochastic {D}ynamic {S}ystems,'' in \emph{2019 American Control Conference
  (ACC)}, 2019, pp. 2958--2963.

\bibitem{shapiro2021lectures}
A.~Shapiro, D.~Dentcheva, and A.~Ruszczynski, \emph{Lectures on stochastic
  programming: modeling and theory}.\hskip 1em plus 0.5em minus 0.4em\relax
  SIAM, 2021.

\bibitem{Markowitz1952}
H.~Markowitz, ``Portfolio {S}election,'' \emph{The Journal of Finance}, vol.~7,
  no.~1, pp. 77--91, Mar. 1952.

\bibitem{Rockafellar1997}
R.~T. Rockafellar and S.~Uryasev, ``Optimization of {C}onditional
  {V}alue-at-{R}isk,'' \emph{Journal of Risk}, vol.~2, pp. 21--41, 1997.

\bibitem{chapman2021classical}
M.~P. Chapman and K.~M. Smith, ``Classical risk-averse control for a
  finite-horizon borel model,'' \emph{IEEE Control Systems Letters}, vol.~6,
  pp. 1525--1530, 2021.

\bibitem{tsiamis2020risk}
A.~Tsiamis, D.~S. Kalogerias, L.~F. Chamon, A.~Ribeiro, and G.~J. Pappas,
  ``Risk-{C}onstrained {L}inear-{Q}uadratic {R}egulators,'' in \emph{2020 59th
  IEEE Conference on Decision and Control (CDC)}.\hskip 1em plus 0.5em minus
  0.4em\relax IEEE, 2020, pp. 3040--3047.

\bibitem{abeille2016lqg}
M.~Abeille, A.~Lazaric, X.~Brokmann \emph{et~al.}, ``Lqg for portfolio
  optimization,'' \emph{Available at SSRN 2863925}, 2016.

\bibitem{Speyer2008STOCHASTIC}
J.~L. Speyer and W.~H. Chung, \emph{Stochastic Processes, Estimation, and
  Control}.\hskip 1em plus 0.5em minus 0.4em\relax Siam, 2008, vol.~17.

\bibitem{ruszczynski2011nonlinear}
A.~Ruszczynski, \emph{Nonlinear Optimization}.\hskip 1em plus 0.5em minus
  0.4em\relax Princeton university press, 2011.

\bibitem{murray}
K.~J. {\AA}str{\"o}m and R.~M. Murray, \emph{Feedback systems}.\hskip 1em plus
  0.5em minus 0.4em\relax Princeton university press, 2010.

\end{thebibliography}
\end{document}